\documentclass[12pt]{amsart}

\usepackage{vmargin}
\usepackage[dvips]{graphicx}
\usepackage{color}
\input{amssym.def}
\input{amssym}
\input{epsf}

\setmarginsrb{3cm}{2cm}{3cm}{3cm}{75pt}{20pt}{20pt}{30mm}
\def\derpar#1#2{\frac{\partial#1}{\partial#2}}
\def\R{\mathbb R}
\def\N{\mathbb N}

\def\Z{\mathbb Z}

\def\E{\mathbb E}
\def\X{\mathcal X}

\def\Sn{{S^{N-1}}} 
\def\nSn{{S^{N-1}(\sqrt{N})}} 
\def\on{{\otimes N}}      

\def\d{{\rm d}}
\def\leb{\cal L}     

\def\var{\varepsilon}

\def\ov{\overline}
\def\cal{\mathcal}

\def\hat{\widehat}

\def\dps{\displaystyle}

\def\med{\medskip}
\def\sm{\smallskip}

\def\begeq{\begin{equation}}
\def\endeq{\end{equation}}
\def\begar{\begin{eqnarray}}
\def\endar{\end{eqnarray}}
\def\begar*{\begin{eqnarray*}}
\def\endar*{\end{eqnarray*}}
\def\begal{\begin{align}}
\def\endal{\end{align}}
\def\begal*{\begin{align*}}
\def\endal*{\end{align*}}

\newtheorem{Thm}{Theorem}
\newtheorem{Lem}[Thm]{Lemma}

\newtheorem{Prop}[Thm]{Proposition}

\newtheorem{Open}[Thm]{Open Problem}
\newtheorem{Ex}[Thm]{Example}

\newtheorem{Quest}{Question}

\theoremstyle{definition}
\newtheorem{Def}[Thm]{Definition}
\newtheorem{Rk}[Thm]{Remark}

\theoremstyle{remark}

\newtheorem*{Thm*}{Theorem}
\newtheorem*{Lem*}{Lemma}
\newtheorem*{Conj*}{Conjecture}
\newtheorem*{Cor*}{Corollary}
\newtheorem*{Def*}{Definition}
\newtheorem*{Prop*}{Proposition}
\newtheorem*{Exo*}{Exercise}
\newtheorem*{Exs*}{Examples}
\newtheorem*{Ex*}{Example}
\newtheorem*{Rk*}{Remark}
\newtheorem*{Rks*}{Remarks}

\def\Ack{\medskip\noindent { {\bf Acknowledgements: } E.C., M.C.C. and M.L. thank the ENS Lyon for hospitality
on a visit during which this work was begun and substantially developed. E.C. and M.L. thank C.M.A.F. and the University of Lisbon for hospitality on a visit during which additional work on this paper was undertaken. The work of E.C. and M.L. was partially supported U.S. National Science Foundation
grant DMS 06-00037. The work of M.C.C. was partially supported by  POCI/MAT/61931/2004. 
}\ \ignorespaces}

\def\signec{\bigskip  \begin{center} {\sc Eric Carlen\par\vspace{3mm}
Department of Mathematics, Hill Center\par
Rutgers University\par
Piscataway, NJ  08854,
U.S.A.\par\vspace{3mm}
e-mail:} \tt{carlen@math.rutgers.edu} \end{center}}

\def\signmcc{\bigskip \begin{center} {\sc Maria Carvalho\par\vspace{3mm}
Department of Mathematics and CMAF\par
 University of Lisbon\par
1649-003 Lisbon,
PORTUGAL\par\vspace{3mm}
e-mail:} \tt{mcarvalh@cii.fc.ul.pt} \end{center}}

\def\signjl{\bigskip \begin{center} {\sc Jonathan Le Roux\par\vspace{3mm}
Department of Information Physics and Computing\par
Graduate School of Information Science and Technology\par
The University of Tokyo\par
7-3-1, Hongo, Bunkyo-ku, Tokyo 113-8656, JAPAN\par\vspace{3mm}
e-mail:} \tt{leroux@hil.t.u-tokyo.ac.jp}\end{center}}

\def\signml{\bigskip  \begin{center}{\sc Michael Loss\par\vspace{3mm} 
School of Mathematics\par
Georgia Institute of Technology\par
Atlanta GA 30332,
U.S.A.\par\vspace{3mm}
e-mail:} \tt{loss@math.gatech.edu} \end{center}}

\def\signcv{\bigskip \begin{center} {\sc C\'edric Villani\par\vspace{3mm}
UMPA, ENS Lyon\par
46 all\'ee d'Italie\par
69364 Lyon Cedex 07,
FRANCE\par\vspace{3mm}
e-mail:} \tt{cvillani@umpa.ens-lyon.fr} \end{center}}

\begin{document}

\title[From microscopic to macroscopic entropy]
{Entropy and Chaos in the Kac Model}

\author{E.A. Carlen}
\author{M.C. Carvalho}
\author{J. Le Roux}
\author{M. Loss}
\author{C. Villani}

\begin{abstract}
We investigate the behavior  in $N$  of the $N$--particle entropy functional for Kac's stochastic 
model of Boltzmann dynamics, and its relation to the entropy function for 
solutions of Kac's one dimensional nonlinear model Boltzmann equation.  We prove a number of
results that bring together the notion of propagation of chaos, which Kac introduced in the context of this model, with the problem of estimating the rate of equilibration in the model in entropic terms, and obtain
a bound showing that the entropic rate of convergence can be arbitrarily slow. Results proved here
show that one can in fact use entropy production bounds in Kac's stochastic model to obtain entropic convergence bounds for his non linear model Boltzmann equation, though the problem of obtaining optimal lower bounds of this sort for the original Kac model remains open, and the upper bounds
obtained here show that this problem is somewhat subtle.

\end{abstract}

\maketitle

\begin{center} {\bf August 22, 2008} \end{center}

\tableofcontents

\section{Introduction}

\subsection{The origins of the problem to be considered}

In a remarkable paper \cite{Kac} of 1956, Mark Kac investigated the probabilistic foundations of kinetic
theory, and defined the notion of {\bf propagation of chaos},
which has since then developed into an active field of probability.

Kac introduced the concept of propagation of chaos in connection with a
specific stochastic process modeling binary collisions in a gas made of
a large number $N$ of identical molecules, and he was particularly concerned with its rate of equilibration; i.e., of its approach to stationarity.

While his ideas concerning propagation of chaos had an immediate resonance and impact, 
this was not the case with the issues he raised concerning rates of equilibration. These had to wait much longer for progress and development, as we shall relate below. In this paper, we bring these two lines of investigation back together, proving several theorems relating chaos and equilibration for the the Kac walk.  

\subsection{The Kac walk}

We begin with a precise description of the {\bf Kac walk}
as a model for the evolution of the distribution of velocities in
a gas of like molecules undergoing binary collisions.
For simplicity, Kac assumed
the gas to be spatially homogeneous, and the velocities $v_j$
($1\leq j\leq N$) to be one-dimensional.
The latter assumption is incompatible with the  conservation of both
momentum and kinetic energy, so Kac only assumed conservation of
the kinetic energy $E$, where
\[ E = \frac m 2 \sum_{j=1}^N v_j^2,\]
with $m$ denoting the mass of the particle species, and $v_j$ denoting the velocity of the 
$j$th particle.

The natural state space for this system  (i.e., state space for the walk) is the sphere
$\Sn(\sqrt{(2/m)E})\subset\R^N$, the $(N-1)$-dimensional sphere with
radius $\sqrt{(2/m)E}$. For the sequel of the discussion, let us choose units in which the mass
of each particle is 2. Then the
total value of the kinetic energy is $N$, so that  the state space is $\nSn$, and each particle
has unit mean kinetic energy. 
 Let $V = (v_1,\dots,v_N)$ denote a generic point in 
 $\nSn$.

Here is how to  take a step of the Kac walk: First, randomly
pick a  pair $(i,j)$ of distinct indices in $\{1,\dots,n\}$ uniformly from among all such pairs.
The molecules $i$ and $j$  are the molecules that will ``collide''. Second,
pick a random  angle
$\theta$ uniformly from  $[0,2\pi)$. Then update $V = (v_1,\dots,v_N)$ by leaving $v_k$
unchanged for $k\ne i,j$, and updating velocities $v_i$ and $v_j$
by rotating in the $v_i,v_j$ plane as follows:
\[ (v_i,v_j)\quad \rightarrow \quad \Bigl((\cos\theta)v_i - (\sin\theta)v_j ,
\:(\sin\theta)v_i + (\cos\theta)v_j\Bigr).\]
Let $R_{i,j,\theta}V$ denote the new point in  $\nSn$ obtained in this way.
This process, repeated again and again, is the {\bf Kac walk} on
$\nSn$.

Associated to the steps of this walk is the Markov transition operator $Q_N$ on 
$L^2(\nSn,\d \sigma^N)$ where $\sigma^N$ is the uniform probability measure on 
$\nSn$. (This notation shall be used throughout the paper.)

If $V_j$ denotes  the position after the $j$th step of the walk, and $\varphi$ is any continuous function
on  $\nSn$, the transition operator $Q_N$ is defined by
$$Q_N\varphi(V) = {\mathbb E}\big\{\phi(V_{j+1}) \ |\ V_j = V\ \big\}\ .$$

From the description provided above, one finds that
$$Q_N\varphi(V) = \left(
\begin{array}{c}
N\\ 2\end{array} \right)^{-1}\sum_{i<j}\ \frac{1}{2\pi}\int_{[0,2\pi)}\varphi(R_{i,j,\theta}V)\d \theta\ .$$
It is easily seen that $\sigma^N$ is the unique invariant measure.

A closer match with the physics being modeled is attained if the steps of the walk arrive 
not in a metronome beat, but in
a Poisson stream with the mean wait between steps being $1/N$.  This ``Poissonification''
of the Kac walk yields a continuous time process on  $\nSn$. Since $Q_N$ is self adjoint on
$L^2(\nSn,\d \sigma^N)$, this process is reversible, and so if the law $\mu_0$ of the initial state
$V_0$ has a density $F^N_0$ with respect to $\sigma^N$, then for all $t>0$, the law
$\mu_t$ of 
$V_t$ has a density $F^N_t$ with respect to $\sigma^N$, and $F^N_t$ is the solution to
the Cauchy problem
\begeq\label{master} 
\frac{\partial}{\partial t}F^N = L_N F^N \qquad{\rm with}\qquad \lim_{t\to 0}F_t^N = F_0^N
\endeq
where $L_N = N(Q_N-I)$, and $I$ is the identity operator.
This equation is known as the {\bf Kac master equation}, which is nothing other than the Kolmogorov forward equation for the continuous time Kac walk. The solution is of course given by
\begeq\label{mastersol} 
F_t = e^{L_N}F_0\ .
\end{equation}

Since $V \to R_{i,j,\theta}$ is a rotation, it follows that
for each positive integer $k$, $Q_N$ preserves the subspace of 
$L^2(\nSn,\d \sigma^N)$ consisting of spherical harmonics of degree no greater than $k$.
Hence, all of the eigenfunctions of $Q_N$ are spherical harmonics. Since the constant is the
only spherical harmonic
that is invariant under  rotations, $1$ is an eigenvalue of
$Q$ of multiplicity one. 

Therefore, for any initial data $F_0$ in $L^2(\nSn,\d \sigma^N)$, the solution 
$F^N_t = e^{L_N}F^N_0$ of the
Kac master  equation satisfies
\begeq\label{masterlim} 
\lim_{t\to\infty}F^N_t = 1\ .
\end{equation}
We refer to the invariant density $1$ as the {\em equilibrium}, and the process of approaching this limit as {\em equilibration}.

The rate at which this limit is achieved is physically interesting for reasons that will be explained shortly.
But apart from its physical motivation, the problem is quite interesting on purely probabilistic grounds:
While the subject of quantifying the rate of equilibration for random walks on large discrete sets
has been vigorously developed in recent years, much less has been done in the case of continuous state spaces  of high dimension, and the Kac walk is a very natural example. 

Kac proposed to investigate the rate of equilibration for his walk in $L^2$ terms through the {\em spectral gap} of $L_N$: Define
$$\Delta_N = \sup\big\{ -\langle \varphi, L_N \varphi\rangle\ : \ \langle \varphi, 1\rangle =0\quad{\rm and}\quad \langle \varphi, \varphi\rangle = 1\ \big\}$$
where the inner products are taken in 
$L^2(\nSn,\d \sigma^N)$. 
In his paper \cite{Kac}, Kac conjectured that 
$\liminf_{N\to\infty} \Delta_N > 0$.

Since one already knows that the eigenfunctions of $L_N$ are spherical harmonics, this may seem
like a trivial problem. In fact, it is very easy to guess the exact value for $\Delta_N$ and the corresponding eigenfunction. Indeed, it is natural to suppose that the eigenfunction must be a simple symmetric, even polynomial in the velocities $v_j$. The simplest such thing,
$\sum_{j=1}^Nv_j^2$, is simply a constant on $\nSn$, so one might try
$$\varphi_{\rm gap} = \sum_{j=1}^N(v_j^4 -\langle v_j^4,1\rangle)\ .$$
(The constant being subtracted to ensure orthogonality to $1$ can be easily computed; see
\cite{CCL:spectralgap:01}, and this is indeed a spherical harmonic.)  Physical reasoning, based on linearizing the Boltzmann-Kac equation to be discussed shortly,  gives further evidence that $\varphi_{\rm gap}$
should in fact be the gap eigenfunction. Using this as a trial function, one readily computes
what should be --- and does turn out to be ---  the value of $\Delta_N$:
\begeq\label{deltaval} 
\Delta_N = \frac{1}{2}\frac{N+2}{N-1}\ .
\end{equation}

However, while one can explicitly compute as many eigenvalues as one wants to,
there is no monotonicity argument to rule out the  proposition that the gap eigenvalue might come from a spherical harmonic of  large degree.

Kac's conjecture that  $\liminf_{N\to\infty} \Delta_N > 0$ was first proved by Janvresse 
~\cite{janvr:kac:01}. Her method did not yield the exact value for $\Delta_N$. 
The first proof that (\ref{deltaval}) is actually correct was given in
\cite{CCL:spectralgap:01}; see also \cite{maslen:kac} for a different approach.
For a treatment of related problems, including 
physical three-dimensional momentum preserving collisions, 
see~\cite{CCL2} and \cite{CGL}.

These results enable us to quantify  (\ref{masterlim}) as follows:
$$\|F^N_t - 1\|_{L^2(\nSn,\d \sigma^N)} \le e^{-t/2}\|F^N_0 -1\|_{L^2(\nSn,\d \sigma^N)}$$
for all $N$ and $t$. While the exponent is uniform in $N$, the shortcoming of this result 
will be familiar to many probabilists who have worked on rates of equilibration: {\em For
natural sequences of initial data $\{F_0^N\}_{N \in \N}$, it will be the case that
$$\|F_0^N\|_{L^2(\nSn,\d \sigma^N)} \ge C^N$$
for some $C>1$.} Therefore, one still has to wait a {\em time proportional to N} before the bound
starts providing evidence of equilibration.

Even worse, the badly behaved sequences of initial data mentioned above are exactly the ones of primary physical interest --- the {\em chaotic} sequences, in which for large $N$ the coordinate
functions $v_j$ are ``nearly independent and identically distributed'' under the law $ \mu^{(N)}= F_0^N\sigma^N$.

\subsection{Kac's notion of chaos}

To state the precise definition, we first introduce some notation that will be used throughout the paper:
Given any probability measure $\mu^{(N)}$ on  $\nSn$,
and any positive integer $k<N$, let $P_k(\mu^{(N)})$ denote the marginal
measure of $\mu^{(N)}$ for the first $k$ velocities. 
In formulas: whenever $A$ is a Borel subset of $\R^k$,
\[ P_k (\mu^{(n)}) [A] = \mu^{(N)} [ \{(v_1,\dots,v_k)\in A\}].\]
In the sequel, we only consider symmetric measures, so there is nothing
particular in considering the {\em first} $k$ velocities.
Chaos means that $P_k\mu^{(N)}$ is well approximated by $\mu^{\otimes k}$,
a distribution of $k$ {\em independent} particles when $N$ is large.
Here is a more precise definition:

\begin{Def}[chaos] \label{defchaos} 
Let $\mu$ be a given Borel probability measure on $\R$. For each positive
integer $N$, let  $\mu^{(N)}$ be a probability measure on $\nSn$. 
Then the sequence of probability measures $\{\mu^{(N)}\}_{N\in \N}$
is said to be {\bf $\mu$-chaotic} in case:
 
 \smallskip
 \noindent{\it (i)} Each $\mu^{(N)}$ is symmetric under interchange of
the variables $v_1,\dots,v_N$;
 
  \smallskip
 \noindent{\it (ii)} For each fixed positive integer $k$, 
the marginal $P_k\mu^{(N)}$ of $\mu^{(N)}$ (marginal on the first $k$
velocities) converges to the $k$-fold tensor product $\mu^{\otimes k}$,
as $N\to\infty$, in the sense of weak convergence against bounded continuous
test functions.  That is, 
whenever $\chi(v_1,\ldots, v_k)$ is a bounded continuous test
function of $k$ variables, then
\begeq\label{cvgcek} 
\int \chi (v_1,\cdots, v_K)\, d\mu^{(N)}(v_1,\cdots, v_N)
\xrightarrow[N\to\infty]{}
\int \chi(v_1,\cdots, v_k)\, d\mu(v_1)\,\cdots\, d\mu(v_K).
\endeq

\end{Def}

Property {\it (ii)} says that $\mu^{(N)}$ is well approximated by
$(P_1\mu^{(N)})^{\on}$ as $N\to\infty$, in the weak sense of convergence
against test functions depending on a finite number of variables.

Besides being archetypal, the following well-known example
will play an important role in this paper. It has quite an ancient history, going back  ---at least ---to Mehler \cite{mehler}
in 1866. For a more recent reference, see \cite{sznit:chaos:91}

\begin{Ex}\label{uniform}
Let
\begeq\label{gammadef}
\gamma(v) = \frac{e^{-v^2/2}}{\sqrt{2\pi}}.
\end{equation}
Then, $\{\sigma^N\}_{N\in \N}$ is $\gamma(v)\d v$ chaotic.
Indeed, this follows easily from the explicit computation
\begeq\label{expligamma}
P_k\sigma^N = \left(1 - \frac{s^2}{N}\right)^{\frac{N-k-2}{2}}\frac{|S^{N-k-1}|}{N^{k/2}|S^{N-1}|}
 \leb_k\ , \quad{\rm where}\quad 
 |\Sn| = \frac{2\pi^{N/2}}{\Gamma(\frac{N}2)}\ ,
\endeq
and where  $\leb_k$ the $k$-dimensional Lebesgue measure. 
\end{Ex}

Now, let $f$ be some probability density on $\R$, and (with the same
notation as in the above example) suppose that $\{F^N\sigma^N\}_{N\in \N}$
is an $f(v)\d v$ chaotic family. For each $N$, let $F^N(t,\cdot)$ denote
the solution of~\eqref{master} at time $t$, starting from the initial
data $F^N$. The main result that Kac did prove in \cite{Kac} is that for each $t>0$,
$\{F^N(t,\cdot)\sigma^N\}_{N\in \N}$ is still a chaotic family;
this property is referred to as {\bf propagation of chaos}.
indeed, $\{F^N(t,\cdot)\sigma^N\}_{N\in \N}$ is $f(t,v)\d v$ chaotic, where
$f(t,v)$ is  the solution of the following Cauchy problem:
\begeq\label{kaceq}
\begin{cases}
f(0,\cdot)=f;\\ \\
\displaystyle{\derpar{f}t (t,v) = 
\frac{1}{2\pi}\int_{-\pi}^\pi\int_\R \Bigl[ f(v',t)\,f(v'_*,t)\ 
-\ f(v,t)\,f(v_*,t)\Bigr]\d v_*\,\d \theta}\ ,
\end{cases}\endeq
and
$$v' = (\cos\theta)\,v \:-\: (\sin\theta)\,v_*;\quad
v'_* = (\sin\theta)\,v \:+\: (\cos\theta)\,v_*\ .$$
This nonlinear equation is a model Boltzmann equation, which we shall
call the {\bf Boltzmann-Kac equation} (as opposed to the Kac master equation). 
The quadratic nonlinearity on the right is a reflection of the fact that
$Q_N$ models a binary collision process, and of Kac's notion of chaos:
Indeed, the time derivative of $P_1(e^{tL_N }F^N)$ may be expressed
in terms of a linear operation on $P_2F^N$, and then this is well approximated
by the tensor product $f\otimes f$ in the limit $N\to\infty$.

The program Kac set forth in \cite{Kac} was to investigate the behavior of solutions of 
(\ref{kaceq}) in terms of
the behavior of solutions of the the Kac master equation (\ref{master}). In particular,
concerning equilibration, 
$$\lim_{t\to\infty} F^N_t =1 \Rightarrow \lim_{t\to\infty}P_1(F^N_t\sigma^N) = \frac{|S^N-2|}{N^{1/2}|S^{N-1}|}\left(1 - \frac{v_1^2}{N}\right)^{(N-3)/2} \approx \gamma(v_1)$$
for large $N$, and thus Kac's theorem can be used to relate the rate of equilibration in the Kac master equation to the rate of convergence in
$$\lim_{t\to\infty}f(v,t) = \gamma(v)$$
for solutions $f(v,t)$ of  (\ref{kaceq}).
Once this would be carried out, one would then like to do the same for the actual Boltzmann equation
for three dimensional velocities with conservation of both energy and momentum. 

As we have indicated, an $L^2$ analysis of the rate of equilibration for solutions of the Kac master equation does not shed much light on the large time behavior of solutions of  (\ref{kaceq}).
What {\em would} do this is very natural in the context of the Boltzmann equation: {\em an entropy
production estimate}.

\subsection{Convergence to equilibrium and entropy inequalities}

If $\mu$ and $\nu$ are two probability measures on a measurable space
$\X$, their relative entropy is defined by the
formula
\[ H(\mu|\nu) = \int h\log h\,d\nu\qquad h=\frac{\d\mu}{\d\nu},\]
with the understanding that $H(\mu|\nu)=+\infty$ if $\mu$ is not
absolutely continuous with respect to $\nu$.
In particular,
\sm

- if $f$ is a probability density on $\R$, then its relative entropy
with respect to $\gamma$ (identified with a probability measure) is
\[ H(f|\gamma) = \int_\R f(v) \log\frac{f(v)}{\gamma(v)}\d v;\]
\sm

- if $F^N$ is a probability density on $\nSn$, then its relative entropy
with respect to the uniform probability measure $\sigma^N$ is
\[ H_N (F^N) := H(F^N\sigma^N|\sigma^N) = 
\int_\nSn F^N (v) \log F^N(v)\,d\sigma^N(v).\]
\sm

The well-known  Csiszar-Kullback-Leibler-Pinsker inequality states that 
\begeq\label{CKP} H(\mu|\nu) \geq \|\mu-\nu\|_{TV}^2/2,
\endeq
where the subscript ``$TV$'' stands for the total variation norm.

So the relative entropy measures a deviation from equilibrium,
just like the $L^2$ norm, and it is natural to try to quantify the rate of equilibration
for the Kac master equation by studying $H_N(F^N_t)$ for solutions:
If $F_t^N$ is a solution,
$$\frac{{\rm d}}{{\rm d}t}H_N(F^N_t) = \int_{\nSn}\log (F_t^N) L_N   F_t^N\d \sigma^N
= \langle  \log(F_t^N) , L_N F_t^N\rangle\ .$$
In analogy with the definition of the spectral gap $\Delta_N$, define the {\em 
entropy production constant} $\Gamma_N$ by
$$\Gamma_N = \inf\frac{- \langle  \log(F^N) , L_N F^N\rangle} {H_N(F^N)}$$
where the infimum is taken over all probability densities $F^N$ on $\nSn$ with 
$H_N(F^N)< \infty$. 

The entropic analog of the Kac conjecture would be that 
there exists a $c>0$ with $\Gamma_N \ge c$ for all $N$. 
This would imply that
\begin{equation}\label{entdec}
H_N(F_t^N) \le e^{-ct}H_N(F_0^N)\ ,
\end{equation}
and hence that
$$\|F_t^N\sigma^N - \sigma^N\|^2_{TV} \le  2e^{-ct}H_N(F_0^N)\ .$$

There is an absolutely crucial difference between this and the $L^2$ 
estimate that we obtained earlier, and it lies in the {\em extensivity of the entropy}.
Suppose that $\{F_0^N\}_{N\in \N}$ is an $f_0(v)\d v$--chaotic family of densities on $\nSn$.
Then according to Kac's theorem,   $\{F_t^N\}_{N\in \N}$ is an $f(v,t)\d v$
--chaotic family of densities on $\nSn$, where $f(v,t)$ is the solution of the Boltzmann--Kac equation with initial data $f_0(v)$. Because of the near product structure of $F_N^t$, 
one might expect that for each $t$, and large $N$,
\begin{equation}\label{extens}
H_N(F_t^N) \approx NH(f(t,\cdot)\d v | \gamma \d v)\ .
\end{equation}
It is the proportionality to $N$ that we refer to as extensivity. Since this factor of $N$ would appear on both sides of (\ref{entdec}) if we substituted (\ref{extens}) in on both sides, we can cancel off the $N$,
and obtain, in the large $N$ limit
$$H(f(t,\cdot)\d v | \gamma \d v)  \le e^{-ct}H(f_0\d v | \gamma \d v)\ .$$
We could  now apply (\ref{CKP}) to this and conclude that for solutions $f(v,t)$ of
the Boltzmann--Kac equation, 
$$\|f(\cdot,t)\d v - \gamma\d v\|^2_{TV} \le 2e^{-ct}H(f_0\d v | \gamma \d v)\ .$$
Such a bound would be very desirable to have for the Boltzmann--Kac equation, and this motivates the enquiry into  the exact behavior of the entropy production constant $\Gamma_N$.

It turns out that estimating the entropy production constant $\Gamma_N$ is a much more subtle
problem than that of estimating the spectral gap $\Delta_N$.  Unfortunately, the best information that is known at present is
$$\Gamma_N \ge \frac{2}{N-1}\ .$$
There are two different proofs of this result. The first, due to Villani, can be found under Theorem 6.1
in \cite{vill:cer:03}. (The bound $2/(N-1)$ is what one gets from the argument in   \cite{vill:cer:03}
making some simplifications that are admissible in the special case of the original Kac model considered here.)
The second, due to Carlen and Loss, is an entropic 
adaptation of the argument used in \cite{CCL:spectralgap:01} to  determine the spectral gap. It can be found under Lemma 2.4 in
\cite{car} using Theorem 2.5 there.   It was conjectured in  \cite{vill:cer:03} that these bounds are essentially sharp; i.e., that
$$\Gamma_N = {\cal O}\left(\frac{1}{N}\right)\ .$$
However, this is not so clear at present. In fact, it had remained an open problem whether there was even a sequence $\{F^N\}$ of  densities for which
\begin{equation}\label{bad}
\lim_{N\to\infty}\frac{- \langle  \log(F^N) , L_N F^N\rangle} {H_N(F^N)} = 0
\end{equation}
with convergence at any rate at all.  The following theorem settles this issue:

\begin{Thm}\label{thmslow}
For each $c>0$, there is a probability density $f$ on $\R$ with $\int_{\R} vf(v)\d v = 0$
and $\int_{\R} v^2f(v)\d v = 1$, and an $f\d v$--chaotic family $\{F^N\sigma\}_{N\in \N}$
such that
$$\limsup_{N\to\infty}\frac{- \langle  \log(F^N) , L_N F^N\rangle} {H_N(F^N)} \le c\ .$$
For each $c$, the density $f$ is smooth and bounded, and has moments of all orders.
\end{Thm}

Once one has this, an easy diagonal argument produces a sequence $\{F^N\}_{N\in \N}$
satisfying (\ref{bad}).

While this would seem to be bad news for Kac's program, it only shows
that one cannot have a {\em universal} bound on the ratio defining $\Gamma_N$, valid for 
{\em all} probability densities $F_N$ on $\nSn$. Theorem  \ref{thmslow} does not rule out the
possibility that there is a {\em conditional} bound on this ratio, holding for all $F^N$ in an $f$--chaotic family with some condition on $f$. 

Indeed, we shall see that the densities $f$ used to prove Theorem \ref{thmslow} have a fourth
moment that diverges as $c$ tends to zero. As far as we now know, Theorem  \ref{thmslow}
might become false under the additional assumption of a fixed bound on the fourth moment of $f$.
This would be very interesting since bounds on the fourth moments are well known to be preserved
by solutions of the Boltzmann--Kac equation, so such a condition on the initial data would propagate.

Moreover, it is known \cite{CL} that even for smooth initial data $f_0$ with 
$\int_{\R} vf(v)\d v = 0$ and $\int_{\R} v^2f(v)\d v = 1$, 
solutions $f(v,t)$ of the Boltzmann--Kac equation can have $\|f(\cdot,t) - \gamma\|_{L^1(\R)}$
approach zero arbitrarily slowly -- for example, like 
$$\frac{1} {1+\log(1+\log(1+ \log(1+t)))}\ ,$$
 or the same thing with as many logarithms as one might wish.  This however can  happen only when  the density $f$
has very long tails so that $\int_{\R} v^2f(v)\d v$ just barely converges. A bound on the fourth moment, which would ensure good behavior of the tails is therefore a plausible condition to impose if one seeks a lower bound on the rate of convergence. 

Finally, if one modifies the Kac walk so that pairs of molecules $i,j$ with high values of
$v_i^2+v_j^2$ run {\em much} faster, then one can prove a uniform positive lower bound on 
$\Gamma_N$; see
~\cite[Section~6]{vill:cer:03}. Thus, Theorem \ref{thmslow}  displays the subtleties that beset   Kac's program, but it does not by any means terminate it.  In fact it raises a very interesting question: 
What sort of conditional bound on $\Gamma_N$ might hold for the Kac model?   But we shall not come to that in this paper; there are more basic issues to be settled first.

\subsection{Conditioned tensor products}

The proof of  Theorem \ref{thmslow} naturally requires the {\em construction of chaotic data}, and this raises the following question:

\begin{Quest}\label{q2}
Let $f$ be a probability density on $\R$
with
\begeq\label{meanvar}
\int_\R vf(v)\,{\rm d}v = 0\qquad,\qquad \int_\R v^2 f(v)\,{\rm d}v = 1\ ,
\endeq
 and finite entropy. Is it true that there
is an $f(v)\d v$--chaotic family of densities $\{F^N\}_{N\in\N}$ on $\nSn$?
\end{Quest}

Question~\ref{q2} may seem trivial at first sight, and actually was
treated by Kac in a rather cavalier fashion. Indeed, there is an obvious
procedure for generating chaotic initial data, which may be described
as follows. 

Suppose that $\mu(\d v)$
is a probability measure on $\R$ satisfying (\ref{meanvar}). 
Consider the tensor product measure $\mu^\on$ and {\em condition} (restrict) 
it to the sphere $\nSn$. By the law of large numbers, 
$\sum_{j=1}^N v_j^2 \approx N$ for large $N$, 
almost surely with respect to $\mu^\on$, so this measure is
roughly concentrated on $\nSn$, and the conditioning should not
modify it too much. 

An important instance where this is obviously
true is the particular case when $\mu=\gamma$: Then $F^N$ is just the
uniform measure $\sigma^N$, and the explicit formula~\eqref{expligamma}
certainly guarantees that $F^N$ is $\gamma$-chaotic in a very strong sense.

But for more general data, the extent to which $\mu^\on$ is actually
concentrated on $\nSn$ is not so obvious. Assume that $\mu$ has
a density $f$, so $f^\on$ is the density of $\mu^\on$; then the
restriction of $f^\on$ to $\nSn$ (which is a set of zero measure)
might just {\em not} be well-defined under  the 
conditions~\eqref{meanvar} alone. Whether or not this is the case depends 
on the fluctuations of $\sum_{j=1}^N v_j^2$ about $N$; i.e., on how well
$\mu^\on$ is concentrated on $\nSn$, as measured by the variance of $v^2$
with respect to $f(v)\d v$.  
Again, {\em this will be governed by a fourth moment condition}.  

In what follows we shall use the following notation: 
For  a probability density $f(v)$ on $\R$, satisfying
$\int f(v)v^2\d v=1$, 
Let $\Sigma^2$ denote the variance of $v^2$ under $f(v)\d v$:
\[ \Sigma:= \sqrt{\int_\R (v^2-1)^2 f(v)\d v}.\]
Also, define
\[ Z_N(f,r):= \int_{S^{N-1}(r)} f^\on\, \d\sigma^N_r,\]
where $S^{N-1}(r)$ is the sphere of radius $r$ in $\R^N$,
and $\sigma^N_r$ is the uniform probability measure on that sphere.

The technical core of our results lies in the following estimates,
that can be seen as a version of the {\bf local central limit theorem}.

\begin{Thm}[Estimates on a conditioned tensor product]\label{propct}
With the above notation and under assumptions (\ref{meanvar}) and 
\begeq\label{addcond} 
\int_\R v^4f(v)\,\d v \ <+\infty\qquad \int_\R f^p <+\infty
\endeq
for some $p>1$,
\begeq
Z_N(f,r) = \gamma^{(N)}(r)\ \frac{\sqrt{2}}\Sigma\
\frac{\alpha_N(N)}{\alpha_N(r^2)}\; \Bigl(e^{-\frac{(r^2-  N)^2}{2N\Sigma^2}}
+ \var(N,f,r) \Bigr),
\endeq
where
$\gamma^{(N)}(r) $
is the restriction of $\gamma^{\on}$ to $S^{N-1}(r)$,
\[ \alpha_N(u) = u^{\frac{N}2-1}e^{-\frac{u}2},\]
and
$\lim_{N\to\infty} \var(N,f,r) = 0$.
\end{Thm}

\begin{Rk} It is part of that Proposition that $Z_N(f,r)$ is well-defined,
at least if $N$ is large enough (it remains unchanged under a modification of
$f$ on a zero Lebesgue measure set).
\end{Rk}

\begin{Rk} We shall prove a more precise version of the theorem,
with explicit estimates on $\var(N,f,r)$; they will be useful
to extend the validity of our results to probability
densities which do not necessarily have finite moment of order~4,
or finite $L^p$ norm. Otherwise, it is sufficient to know that
$\var(N,f,r)\to 0$ as $N\to\infty$.

\end{Rk}

The implications of Proposition~\ref{propct} are best understood when
recast in terms of the relative density of $f$ with respect to $\gamma$;
so let
\begeq\label{Z'N} 
Z'_N(f,r):= \int_{S^{N-1}(r)} \left(\frac{f}{\gamma}\right)^\on\, 
\d\sigma^N_r.
\endeq
Then, as a consequence of Proposition~\ref{propct},
\[ Z'_N(f,\sqrt{N}) = \frac{\sqrt{2}}{\Sigma} \bigl(1+o(1)\bigr).\]
\medskip
\noindent{$\bullet$} {\it
In other words, the integral of $f^\on$ on $\nSn$ has a universal
behavior -- depending on $f$ only through $\Sigma$.}
\medskip

Thus, the fourth moment condition in Theorem \ref{propct} is just what is required, in the way of moments,
 for the conditioning to work. What about the $L^p$ condition?

This comes in as follows: As a function of $r$, $Z_N(f,r)$ can be expressed in terms of the
density for $\sum_{j=1}^N V_j^2$, where $\{V_j\}_{j\in \N}$ is a sequence
of independent random variables with law $f(v)\d v$.
By Young's convolution inequality, the $N$--fold convolution power
of a  probability density $g$ is continuous if $g$ lies in $L^{N/(N-1)}$.
Hence the $L^p$ condition in Proposition~\ref{propct}
is natural: It is a simple sufficient condition to ensure that
$Z_N(f,r)$ is a continuous function of $r$ if $N$ is large enough.  
Interestingly enough, though $p$ can be arbitrarily close to $1$, a bound on the entropy is {\em not} enough to ensure this.
This point is discussed further in the appendix where we prove the version of the local central limit 
theorem that we shall use here.

When the conditions of  Theorem \ref{propct} are satisfied, we may condition  the tensor product
$\mu^{\otimes N}$, with $\mu = f\d v$, to obtain a probability measure on $\nSn$:

\begin{Def}[Conditioned product measures] \label{defcpd} 
Given a probability density $f$ on $\R$ satisfying the hypotheses of 
Proposition~\ref{propct}, and $\mu(\d v)=f(v)\d v$,
we define the corresponding family
of {\em conditioned product measures}, denoted
$\{ \ [\mu^{\on}]_{\nSn}\ \}_{N\in \N}$, by
\[[\mu^{\on}]_{\nSn} := 
\frac{\prod_{j=1}^Nf(v_j)}{Z_N(f,\sqrt{N})} \,\sigma^N
= \frac{\prod_{j=1}^N(f(v_j)/\gamma(v_j))}{Z'_N(f,\sqrt{N})} \,\sigma^N\ .\]
\end{Def}

The point of this definition is that, as noted above, one might hope that the family
$\{ \ [\mu^{\on}]_{\nSn}\ \}_{N\in \N}$ would be $\mu$--chaotic. This is indeed the case, and in a very strong sense, as we shall explain in the next subsection.

\subsection{Entropic chaos}

The notion of chaos as originally defined by Kac is well adapted to his original purpose, namely,
establishing a rigorous  connection between the linear Kac master equation on the one hand, and the non linear Boltzmann--Kac equation on the other. However, it is not quite strong enough
to draw conclusions about the entropic
rate of convergence to equilibrium for the Boltzman--Kac equation from an analysis of the entropic rate of convergence for the Kac master equation.  As we have explained above, a rigorous deduction in this direction would depend on having a precise version of the extensivity property (\ref{extens})
for chaotic families.  Thus we ask:

\begin{Quest}\label{q1}
Is there a reformulation of the chaos property in entropic terms
that is sufficiently strong that it can yield a bound on the entropic
rate of convergence to equilibrium for~\eqref{kaceq} when combined with
a bound on the entropic rate of convergence for~\eqref{master}?
\end{Quest}

As we shall see, the answer is positive:

\begin{Def}[Entropic $\mu$-chaos] \label{defschaos} 
Let $\mu$ be a probability measure on $\R$, and, 
for each positive integer $N$, let $\mu^{(N)}$ be a probability measure
on $\nSn$. The sequence $\{\mu^{(N)}\}_{N\in \N}$ is
said to be entropically $\mu$-chaotic in case it satisfies conditions
$(i)-(ii)$ in Definition~\ref{defchaos}, and in addition

\smallskip
\noindent{\it (iii)}  \qquad\qquad\qquad\qquad
${\displaystyle \lim_{N\to \infty} \frac{H(\mu^{(N)}|\sigma^N)}{N}
= H(\mu|\gamma)}$.

\end{Def}

As indicated above, from the physical point of view, condition $(iii)$
can be thought of as expressing {\it asymptotic extensivity}
of the entropy for an entropically  chaotic family; it  provides a
bridge between the entropy of the $N$-particle system and the
entropy of the reduced system. This is reminiscent of a work by Kosygina
on the limit from microscopic to macroscopic entropy in the Ginzburg-Landau
model~\cite{kosy:excl:01}.

Secondly,  entropic chaos really is a stronger notion than plain
chaos;  it involves {\em all}  of variables, not
only  finite-dimensional marginals of fixed size. There is a good analogy with a
work by Ben Arous and Zeitouni~\cite{GBAZ:increasing:99}
(also based on the extensivity properties of entropy).
Their work, just as ours, uses a version of the Central Limit Theorem. 

Finally, once Condition {\em (ii)} is enforced, Condition {\em (iii)} really means
that $\mu^{(N)}$ is ``strongly'' close to $\mu^\on$. To understand this,
think of the following well-known theorem: If $f^{(N)}$ is a family of
probability densities on $\R$, converging weakly to some probability
density $f$ as $N\to\infty$, and $J$ is a strictly convex functional,
then automatically $J(f)\leq \liminf J(f^{(N)})$; but if in addition
$J(f^{(N)})\to J(f)$, then the convergence of $f^{(N)}$ to $f$
actually holds almost in the sense of $L^1$ norm, 
not just weakly. So one could define
a notion of strong convergence by requiring the weak convergence
of $f^{(N)}$, plus the convergence of $J(f^{(N)})$ to $J(f)$.
Such a step has already been taken in the definition of
the ``entropic convergence'' used in the context of (deterministic)
hydrodynamic limits of the Boltzmann equation by Golse and collaborators,
in an impressive series of papers, starting with~\cite{BGL:1:91}
and leading up to~\cite{GSR:NS:04}.

The following theorems provides an answer to both Questions \ref{q2} and \ref{q1}:

\begin{Thm}\label{thmsuffschaos}
Let $f$ be a probability density on $\R$ satisfying
\[ \int f(v) v^2\, \d v=1\qquad \int f(v)v^4\,\d v<+\infty,\qquad
f\in L^\infty(\R),\]
and let $\mu(\d v)=f(v)\d v$. Then $\{[\mu^{\on}]_{\nSn}\}_{N\in \N}$ is entropically $\mu$--chaotic.
In fact, condition (ii) from the definition of chaos holds in the following much stronger sense:
\begin{equation}\label{strong2}
\lim_{N\to\infty} H\bigl(P_k([\mu^{\on}]_{\nSn})\ \big|\ \mu^{\otimes k}\bigr)  = 0\ .
\end{equation}

Furthermore,  let $\{\mu^{(N)}\}_{N\in \N}$ be any family of symmetric probability measures on $\nSn$
such that

\noindent{\it (H)} $\dps \qquad\qquad\qquad
\frac1N H\bigl(\mu^{(N)}\ \big|\ [\mu^{\on}]_{\nSn} \bigr)\xrightarrow[N\to\infty]{}0\ .$
\smallskip
Then $\{\mu^{(N)}\}_{N\in\N}$ is entropically $\mu$-chaotic. 
\end{Thm}

Theorem \ref{thmsuffschaos} takes care of Question \ref{q2} for bounded densities $f$
with a finite fourth moment, but certainly we cannot directly employ the conditioned tensor product construction when  $f$ does not have a fourth moment. However, using a diagonal argument, we shall be able to show that there does exist an entropically $f\d v$--chaotic family for all finite energy, finite entropy probability densities
$f$ on $\R$:

\begin{Thm}\label{thmexist}
Let $f$ be a probability density on $\R$ with
\[ \int f(v)v^2\d v =1, \qquad H(f|\gamma)<+\infty.\]
Then there exists an $f(v)\d v$-entropically chaotic sequence.
\end{Thm}

Theorem \ref{thmsuffschaos} has the following shortcoming: One might hope that
 for {\em any} $f\d v$--entropically chaotic family $\{\mu^{(N)}\}$, and not only conditioned tensor products, one would have

\medskip
\noindent  ($ii^\prime$)\ For any $k\in\N$, ${\displaystyle 
\lim_{N\to\infty} H(P_k(\mu^{(N)}|\mu^{\otimes k})  = 0\ .}$
\medskip

We would then include condition $(ii')$ in the definition of entropic chaos.  However, Theorem \ref{thmsuffschaos}  asserts this only when  $\{\mu^{(N)}\}$ is a conditioned tensor product. While the
set of conditioned tensor product states is not propagated into itself by the Kac mater equation,
probability densities satisfying condition $(H)$ for some
$\{[\mu^{\on}]_{\nSn}\}_{N\in \N}$ may well be.
This leads to the following problem, for which we have no solution:

\begin{Open} \label{openpb}
Does Condition $(H)$ in Theorem~\ref{thmsuffschaos}
also imply Condition   ($ii^\prime$)?  More generally, does  ($ii^\prime$) hold for a larger and easily recognized class
of chaotic sequences, larger than those constructed by means of 
conditioning tensor products?
\end{Open}

Also, as indicated above, a natural next step in the development of 
Kac's program consists in studying the
propagation of Conditions {\em (iii)}  or $(H)$  --- or  {\em ($ii^\prime$)} 
--- under the Kac master equation.

\subsection{Final remarks}

As an intermediate step in the proof of Theorem~\ref{thmsuffschaos},
we shall establish the following statement:

\begin{Thm}[asymptotic upper semi-extensivity of the entropy]
\label{thmuppergen}
For each $N$, let $\mu^{(N)}$ be a probability density on
$\nSn$, such that $\mu^{(N)}$ is $\mu$-chaotic, in the sense of
Definition~\ref{defchaos}. Then
\[ H(\mu|\gamma) \leq \liminf_{N\to\infty}
\frac{H(\mu^{(N)}|\sigma^N)}N.\]
\end{Thm}

This result certainly has interest in its own right, and
further explains the meaning of Condition $(H)$ in Theorem~\ref{thmsuffschaos}.
By the way, Theorem~\ref{thmuppergen} and Theorem~\ref{thmexist} together
provide a proof of Remark~2 following Theorem~6.1 in~\cite{vill:cer:03}.
(The author had at the time thought that this remark was obvious.)
This combination of results also establishes a kind of
$\Gamma$-convergence of the functionals $H(\cdot|\sigma^N)/N$ 
to the functional $H(\cdot|\gamma)$.
\med

We close our introductory discussion with some final remarks on Kac's program.
Kac suggested that one could prove quantitative theorems on the non linear
Boltzmann-Kac equation by means of an investigation of the linear
master equation. At the time Kac wrote his paper, the rigorous mathematical
theory of the Boltzmann equation had been in the doldrums since the landmark work of Carleman
\cite{carleman} in the thirties. The suggestion of Kac to recast the problem 
of investigating  nonlinear equations such as~\eqref{kaceq} from a {\em probabilistic}
many particle point of view was made in the hope that this might be a better path to progress.

However, the history of the subject has
not developed as Kac had hoped.  The lack of progress between the papers of
Carleman and Kac turned out to be due as much to lack of attention as to the intrinsic difficulties of
nonlinear equations equations such as~\eqref{kaceq}. Once a new generation of mathematicians took
up such equations as an active field of research, a well developed and full--fledged theory emerged.
And so far, no relevant property of the nonlinear
equation~\eqref{kaceq} has been proved via~\eqref{master},
which cannot be proved by  direct means.  Indeed, once again in this paper, {\em we shall prove
lack of a uniform entropy production inequality for the Kac master equation (Theorem \ref{thmslow}) through an analysis of the Boltzmann--Kac equation}.

Still, Kac's program
is worth trying to push for various reasons. 
First, the theory of
spatially homogeneous Boltzmann equations has reached maturity,
with quite precise results, and specialists are now looking for
very sharp statements; it might be that Kac's approach, thanks to its
strong physical content, could be adapted to such refinements. Just because so far
no relevant property of ~\eqref{kaceq} has been first proved via~\eqref{master} does not mean that this is
cannot be done, and certainly the probabilistic ground is less worked--over. 

Second, it can be seen as a baby model for the much more subtle
problem of propagation of chaos in the ``true'' spatially
inhomogeneous Boltzmann equation. 

Finally, one might be interested
in it for just historical reasons, since Kac's paper is one
of the founding works in modern kinetic theory ---
and just perhaps, the renewed focus on Kac's ideas
will yield new progress of a fundamental sort.

\subsection{Organization of the paper}

In Section~\ref{secrestr} below, we first study the
asymptotics of the restricted tensor product, and prove
Proposition~\ref{propct}. In Section~\ref{secupper},
we establish the asymptotic upper semi-continuity of the entropy
(Theorem~\ref{thmuppergen}). In Section~\ref{secmarginals},
we study the convergence of marginals, establishing in particular 
Condition~(ii) of Definition~\ref{defchaos} for the restricted tensor product.
Asymptotic extensivity of the restricted tensor product
(or perturbations thereof) will be proven in Section~\ref{secmicromacro}.
Then, in Section~ \ref{secunbounded} we prove
Theorem~\ref{thmexist}.
Finally, we shall investigate entropy production and prove Theorem 
\ref{thmslow} in Section \ref{production}.  The  appendix contains the statement
and proof of a version of the local central limit theorem with precise quantitative
bounds that we require in   Section~\ref{secrestr}, but it also has
some independent interest.

We close  this introduction by thanking Julien Michel 
for providing reference~\cite{FoataFuchs}; and Alessio Figalli 
for his careful reading of and comments on an earlier version of the 
manuscript.

\section{Asymptotics of the restricted tensor product} \label{secrestr}

The goal of this section is to analyze the asymptotic behavior of
\[ Z_N(f,r):= \int_{S^{N-1}(r)} f^\on\,\d \sigma^N_r\]
as $N\to\infty$, where $\sigma^N_r$ is the uniform probability measure
on $S^{N-1}(r)$.

\begin{Lem}[probabilistic interpretation of $Z_N$] \label{lemZn}
Let $f$ be a probability density on $\R$, and 
let $\{V_j\}_{j\in \N}$ be a sequence of independent random variables
with common law $f(v)\d v$. Then the random variable
$S_N:=\sum_{j=1}^N V_j^2$ has density $s_N(u)\,\d u$, where
\begeq\label{sden}
 s_N(u)= \frac{|\Sn|}2 \, u^{\frac{N}2 -1} Z_N(f,\sqrt{u}) \ ,
 \endeq
where
${\displaystyle  |\Sn| = \frac{2\pi^{N/2}}{\Gamma(\frac{N}2)}}$
stands for the $(N-1)$-dimensional volume of the unit sphere $\Sn\subset\R^N$.
As a particular case, the law of $V_1^2$ has density $h(u)\,\d u$,
where
\begeq\label{handf} 
h(u) = \frac{1}{2\sqrt{u} }\Bigl(f(\sqrt{u}) + f(-\sqrt{u})\Bigr).
\endeq
\end{Lem}

\begin{proof}[Proof of Lemma~\ref{lemZn}]
We use the notation $r=\sqrt{\sum v_i^2}$ and let $\E$ denote the expectation 
with respect to the uniform probability measure on $\Sn$.
Whenever $\varphi$ is a continuous test function supported in
$[0,+\infty)$, a polar change of variables leads to
\begin{align*}
\E\: \varphi \left(\sum_{j=1}^N V_j^2\right) & = 
\int_{\R^n} f^\on (v)\, \varphi(r^2)\,\d v 
 = |\Sn| \int_{[0,+\infty]\times \Sn} f^\on \varphi(r^2) r^{N-1}\,\d r\,\d\sigma \\
& = |\Sn| \int_0^{+\infty} \varphi(u) \left (
\frac{u^{\frac{N-1}2}}{2\sqrt{u}} \int_{\Sn} f(\sqrt{u}\, y_1) \ldots
f(\sqrt{u}\, y_N)\,d\sigma(y) \right )\,\d u \\
& = \int_0^{+\infty} \varphi(u) \left ( \frac{|\Sn|}2 \, u^{\frac{N}2 -1}
Z_N(f,\sqrt{u}) \right )\,du.
\end{align*}
\end{proof}

Our next theorem, which is the main result of this section,  is a slightly sharpened
version of Theorem~\ref{propct}. It provides more information on how the remainder terms there depend on $f$. In describing this dependence, we shall use the following notation:

Let $f$ be a probability density on $\R$ with finite moment of order~4,
and finite $L^p$ norm for some $p\in (1,\infty)$.
Define the mean kinetic energy and its variance by
\[ E =  \int f(v)v^2\,\d v;\qquad
\Sigma = \sqrt{\int_\R (v^2-E)^2 f(v)\,\d v}.\]
(As is in the introduction, we have chosen units in which the mass $m$ is
equal to~$2$.) Let $\underline{E}$, $\ov{E}$, $\underline{\Sigma}$,
$\ov{L}$ be constants such that
\[ 0<\underline{E}\leq E\leq \ov{E}<+\infty;\qquad
\Sigma\geq \underline{\Sigma}>0;\qquad
\|f\|_{L^p} \leq \ov{L},\]
and let $\chi_4$ be {\em any} nonnegative function of $r>0$, such that
$\chi_4(r)\longrightarrow 0$ as $r\to 0$ and
\[ \int_{|v|\geq \frac1r} f(v)v^4\,\d v \leq \chi_4(r).\]
(For instance, we could define $\chi_4$  to be the left hand side, but in applications another choice, such as
${\displaystyle r^2\int_{\R}f(v)v^6\d v}$ may be more useful if, say, $f$ possesses a $6$th moment.)

\begin{Thm}[asymptotics for the conditioned tensor product]\label{thmasymp}
With the above notation, define
\[ Z_N(f,r):= \int_{S^{N-1}(r)} f^\on\, \d\sigma^N_r,\]
where $S^{N-1}(r)$ is the sphere of radius $r$ in $\R^N$,
and $\sigma^N_r$ is the uniform probability measure on that sphere.
Then, as $N\to\infty$,
\[ Z_N(f,r) = \frac{\sqrt{2}}{\Sigma}\: \gamma^{(N)}(r)
\left(\frac{\alpha_N(N)}{\alpha_N(r^2)}\right)
\Bigl[ e^{-\frac{(r^2-NE)^2}{2N\Sigma^2}} + o(1) \Bigr],\]
where 
\[ \gamma^{(N)}(r) = \frac{e^{-r^2/2}}{(2\pi)^{N/2}}\]
is the restriction of $\gamma^{\on}$ to $S^{N-1}(r)$, 
\[ \alpha_N(s) = s^{\frac{N}2-1} e^{-\frac{s}2},\]
and $o(1)$ stands for an expression which is bounded
by a function $\omega(N)\to 0$, depending only on
$\underline{E}$, $\ov{E}$, $\underline{\Sigma}$,
$p$, $\ov{L}$ and $\chi_4$.

In particular,
\[ Z_N(f,\sqrt{N}) = \frac{\sqrt{2}}{\Sigma}
\gamma^{(N)}(\sqrt{N}) \Bigl( e^{-\frac{N(1-E)^2}{2\Sigma^2}} + o(1) \Bigr).\]
\end{Thm}

\begin{proof}[Proof of Theorem~\ref{thmasymp}]
Since (with the notation of Lemma~\ref{lemZn})
$V_1^2$ has density $h$, it follows that $S_N$ has density
$h^{\ast N}$ (the $N$-fold convolution product of $h$ with itself).
So $s_N=h^{\ast N}$, which leads to the formula
\begeq\label{ZNfu} 
Z_N (f,\sqrt{u}) = \frac{2 h^{\ast N}(u)}{u^{\frac{N}2-1} |\Sn|}.
\endeq

We shall use the local central limit theorem to approximate
$h^{\ast N}$. For that we need some estimates on $h$.
First note that
$\int h(u)\,u\,\d u = E$ and $\int h(u)\, u^2\,\d u =
E^2 + \Sigma^2$
Also,
\[ \int_{u\geq 1/r} h(u) u^2\,\d u =
\int_{v\geq 1/\sqrt{r}} f(v) v^4\,\d v.\]
Next, let $q>1$; by convexity of $t\mapsto t^q$, and the definition of $h$,
\begin{align*}
\int_{\R_+}h^q(u)\,\d u &
\le \frac12 \int_{\R_+} u^{-q/2}\bigl(f^q(\sqrt{u}) + 
f^q(-\sqrt{u})\bigr)\d u\\
&= \int_{\R_+} u^{-(q-1)/2}\; \frac{1}{2\sqrt{u} }
\bigl(f^q(\sqrt{u}) + f^q(-\sqrt{u})\bigr)\d u \\
&= \int_\R |v|^{1-q}f^q(v)\d v\ \leq
\int_{[-1,1]} \frac{f^q(v)}{|v|^{q-1}}\,\d v \ + \
\int f^q(v)\d v.
\end{align*}
If $q<(2p)/(p+1)$, then, by H\"older's inequality,
\[ \int_{[-1,1]} \frac{f^q(v)}{|v|^{q-1}}\,\d v \ \leq
\left(\int f^p\right)^{\frac{q}{p}}
\left(\int_{[-1,1]}\frac{\d v}{|v|^{\frac{(q-1)p}{p-q}}}\right)
^{1-\frac{q}{p}}\leq C(p,q)\|f\|_{L^p}^q.\]
On the other hand, $\int f^q\d v= \|f\|_{L^q}^q \leq \|f\|_{L^p}^q$
as soon as $q\leq p$, because $f$ is a probability measure.
The conclusion is that  there is a finite constant
$C(p,q)$ such that
\begeq\label{lqbnd}
\|h\|_{L^q} \leq C(p,q) \|f\|_{L^p} \qquad {\rm for\  all}\qquad  q<(2p)/(p+1)\ .
\endeq

Now let $g$ be defined by
$g(v) = \Sigma \, h(E + \Sigma v)$.
so that
\[ \int g(v)\,\d v =1 \qquad
\int g(v)\,v\,\d v=0\qquad 
\int g(v)\,v^2\,\d v=1\ .\]
It follows immediately from (\ref{lqbnd}) that $g$ lies in $L^q$ for 
some $q>1$.

Also, $g$ inherits from $h$ a sort of ``concentration bound'' 
that we require to apply Theorem~\ref{LCLT} in Appendix~\ref{secLCLT}:
\begin{align*} \int_{u\geq \frac1r} g(u) u^2\,\d u
& = \frac1{\Sigma} \int_{s\geq \Sigma/r +E} h(s) \,
(s-E)^2\,\d s \\
& \leq \frac2{\Sigma} \int_{s\geq \Sigma/r +E}
h(s) (s^2+E^2)\,\d s  \\
& \leq \frac{2}{\Sigma} \int_{s\geq \Sigma/r +E}
h(s) s^2\,\d s \ + \ 
 \frac{2}{\Sigma}\frac{E^2 + \Sigma^2}{(\Sigma/r +E)^2}.
\end{align*}
Evidently,  the quantity on the left goes to~0 as $r\to 0$,
with a rate which depends on $\int_{|v|\geq 1/\sqrt{r}} f(x)x^4\,\d x$.

As a conclusion, $g$ satisfies all the assumptions of
Theorem~\ref{LCLT} in Appendix~\ref{secLCLT},
so there is a function $\lambda(N)$,
only depending on the above-mentioned bounds, such that
\[ \sup_{u\in\R} \bigl|\sqrt{N} g^{\ast N} (\sqrt{N} u) - \gamma(u) \bigr| 
\leq \lambda(N)\ ,\]
and so,
\[ \sup_{u\in\R} \Bigl| g^{\ast N}(u) - \frac{1}{\sqrt{N}}\gamma \left(\frac{u}{\sqrt{N}}\right)
\Bigr| \leq \frac{\lambda(N)}{\sqrt{N}}.\]
Then, since 
$g^{\ast N}(u) = 
\Sigma h^{\ast N}(NE+\Sigma u)$,
we deduce
\[ \sup_{x\in\R}\: \biggl|
h^{\ast N}(x)\ -\ \frac{1}{\sqrt{N}\Sigma}\, \gamma\left(
\frac{x-NE}{\sqrt{N}\Sigma}\right) \biggr|
\leq \frac{\lambda(N)}{\sqrt{N}\Sigma}.\]

Now let us insert this bound in~\eqref{ZNfu} and apply Stirling's formula,
in the form
\[ \Gamma\left(\frac{N}2\right) =
\sqrt{\pi N}\, \alpha_N(N)\, 2^{-\frac{N}2 +1} 
\Bigl(1+O\left(\frac1{\sqrt{N}}\right)\Bigr),\]
where
$\alpha_N(u):= u^{\frac{N}2-1}e^{-u}$.
This results in
\[ Z_N(f,\sqrt{u}) = 
\frac{\sqrt{\pi N} \alpha_N(N)\, 2^{-\frac{N}2+1}}{\sqrt{N}\Sigma}
\, \left(\gamma\left(\frac{u-NE}{\sqrt{N}\Sigma}\right) + o(1)\right)
\left(1 + O\left(\frac1{\sqrt{N}}\right)\right).\]
Now the desired expression follows easily.
\end{proof}

\begin{Rk} Consider the case when $E=1$; then
\[ Z_N(f,\sqrt{N}) = \gamma^{(N)}(N) \frac{\sqrt{2}}{\Sigma}.\]
Thus, after renormalization by the Gaussian density, $Z_N(f,\sqrt{N})$
has a nontrivial finite limit as $N\to\infty$. This was derived in
the above proof by  using successively the
Local Central Limit Theorem and Stirling's formula.
Actually, the latter can be eliminated: Since $\gamma^\on$ is constant on
the sphere (by the way, the Gaussian is the only tensor product function
to satisfy this property), we can write, with the notation~\eqref{Z'N},
\[ Z'_N(f,r) = \frac{\dps\int_{S^{N-1}(r)} f^\on\,\d\sigma^N_r}
{\dps\int_{S^{N-1}(r)} \gamma^\on\,\d\sigma^N_r}
= \frac{Z_N(f,r)}{Z_N(\gamma,r)},\]
and in view of Lemma~\ref{lemZn} this simplifies into
\[ Z'_N(f,r)=\frac{h_N(f,r^2)}{h_N(\gamma,r^2)},\]
the asymptotics of which can be computed by using only the Local
Central Limit Theorem, and not Stirling's formula. Actually,
this line of reasoning shows how to \emph{deduce} Stirling's formula
from the Local Central Limit Theorem. This observation in itself
is not new (it is made explicitly in the discussion of the 
Local Central Limit theorem in~\cite{FoataFuchs}; 
see also~\cite[Problem~2]{major:problemsfourier}).
However, it is
interesting to see that here it arises naturally as part of a physically
relevant problem.
\end{Rk}

We conclude this section with a higher-dimensional generalization
of Theorem~\ref{thmasymp}, beginning with the analogs of some definitions we made for densities on $\R$.
. Let $f$ be a probability density on
$\R^k$ with finite moment of order~4, and finite $L^p$ norm
for some $p\in (1,\infty)$. Define
\[ E =  \int_{\R^k} f(v)|v|^2\,\d v;\qquad
\Sigma = \sqrt{\int_{\R^k} (|v|^2-E)^2 f(v)\,\d v}.\]
Let $\underline{E}$, $\ov{E}$, $\underline{\Sigma}$,
$\ov{L}$ be such that
\[ 0<\underline{E}\leq E\leq \ov{E}<+\infty;\qquad
\Sigma\geq \underline{\Sigma}>0;\qquad
\|f\|_{L^p} \leq \ov{L},\]
and let $\chi_4$ be a nonnegative function of $r>0$, such that
$\chi_4(r)\longrightarrow 0$ as $r\to 0$ and
\[ \int_{|v|\geq \frac1r} f(v)|v|^4\,\d v \leq \chi_4(r).\]
Again in the higher dimensional case, it is easy to see that 
if $f\in L^p$ for some $p>1$, then the density  $h$ of  $v^2$ 
under $f(v){\rm d}v$ is in $L^q$ for some $q>1$. 
Then we have the following analogue of Theorem~\ref{thmasymp}:

\begin{Thm}[asymptotics for the conditioned tensor product]\label{thmasympk}
With the above notation, define
\[ Z_m(f,r):= \int_{S^{km-1}(r)} f^{\otimes m}\, \d\sigma^{km}_r,\]
Then, as $m\to\infty$,
\[ Z_m(f,r) = \frac{\sqrt{2}}{\Sigma}\: \frac{e^{-\frac{r^2}2}}
{(2\pi)^{\frac{km}2}} \left(\frac{\alpha_{km}(km)}{\alpha_{km}(r^2)}\right)
\Bigl[ e^{-\frac{(r^2-mE)^2}{2m\Sigma^2}} + o(1) \Bigr],\]
where $o(1)$ stands for an expression which is bounded
by a function $\omega(m)\to 0$, depending only on
$k$, $\underline{E}$, $\ov{E}$, $\underline{\Sigma}$,
$p$, $\ov{L}$ and $\chi_4$.

In particular,
\[ Z_m(f,r) = \frac{\sqrt{2}}{\Sigma} \gamma^{(km)}(r) 
\left(\frac{\alpha_{km}(km)}{\alpha_{km}(r^2)}\right)
\Bigl( e^{-\frac{m(k-E)^2}{2\Sigma^2}} + o(1) \Bigr).\]
\end{Thm}

\begin{proof}[Sketch of a proof of Theorem~\ref{thmasympk}]
First, one can adapt the proof of Lemma~\ref{lemZn} to the
present case; the conclusion should be changed as
follows: $S_m:=\sum_{j=1}^m |V_j|^2$ has density $s_m(u)\,\d u$,
where
\[ s_m(u) = \frac{|S^{km-1}|}2\: u^{\frac{km}2-1}\,
Z_m(f,\sqrt{u}).\]
In particular, the law of $|V_1|^2$ is
\[ h(u) = \frac{|S^{km-1}|}2\: u^{\frac{km}2-1}
\int_{S^{k-1}(\sqrt{u})} f\,\d\sigma^k_{\sqrt{u}}.\]

Then the proof of Theorem~\ref{thmasymp} adapts to the present case
with hardly any change, upon replacement of $N$ by $mk$.
\end{proof}

\section{Asymptotic upper semi-continuity of the entropy}\label{secupper}

To motivate this section, let us recall an important property of the
entropy functional. Consider a sequence of probability measures
$\mu^N$ on $\R$, converging weakly to another probability distribution
$\mu(\d v)$ as $N\to\infty$, and let $\nu$ be another probability measure.
Then
\[ H(\mu|\nu) \leq \liminf_{N\to\infty} H(\mu^N|\nu).\]
In other words, the relative entropy is {\em lower semi-continuous}
under weak convergence. 

In this section, we shall show that the same property is true
when the dimension goes to $\infty$, and weak convergence is
replaced by the chaos property. The following theorem is
a generalization of Theorem~\ref{thmuppergen}.

\begin{Thm}\label{thmuppergen'}
Let $g$ be a probability density on $\R$, such that
\[ \int g(x)x^2\,\d x=1,\qquad 
\int g(x) x^4\,\d x<+\infty,\qquad
g\in L^p(\R)\quad (p>1);\]
define $\nu(\d v)=g(v)\,\d v$. For each positive integer $N$,
let  $\mu^{(N)}$ be a symmetric probability measure on $\nSn$,
such that
\[ P_1\mu^{(N)} \xrightarrow[N\to\infty]{} \mu,\]
in the sense of weak convergence against bounded continuous functions. Then
\[ H(\mu|\nu) \leq \liminf_{N\to\infty} \ 
\frac{H(\mu^{(N)}|[\nu^{\on}]_{\nSn})}{N}.\]

More generally, if, for some positive integer $k$,
\[ P_k\mu^{(N)} \xrightarrow[N\to\infty]{} \mu_k,\]
then
\[ \frac{H(\mu_k|\nu^{\otimes k})}k \leq \liminf_{N\to\infty} \ 
\frac{H(\mu^{(N)}|[\nu^{\on}]_{\nSn})}{N}.\]
\end{Thm}
\med

The proof uses the results of Section~\ref{secrestr}, plus
a duality formula for the entropy:

\begin{Lem}[Legendre representation of the $H$ functional]
\label{lemlegH}
Let $\X$ be a locally compact complete metric space equipped with
a reference (Borel) probability measure $\nu$. Then,
for any other probability measure $\mu$ on $\X$,
\begeq\label{Hleg1}
H(\mu|\nu) = 
\sup\: \left\{\int \varphi\,\d\mu - 
\log\left( \int e^{\varphi}\,\d\nu\right); \ \varphi \in C_b(\X)\right\},
\endeq
where $C_b(\X)$ stands for the space of bounded continuous functions
on $\X$. 

Moreover, one can restrict the supremum in~\eqref{Hleg1} 
to those functions $\varphi$ such that $\int e^{\varphi}\,\d\nu=1$.
\end{Lem}

We skip the proof of Formula~\eqref{Hleg1}, which belongs to folklore
(see e.g.~\cite[Appendix~B]{LV} for a complete proof in the case
of a compact space $\X$). As for the last part of Lemma~\ref{lemlegH}, 
it follows from an easy homogeneity argument.

\begin{proof}[Proof of Theorem~\ref{thmuppergen'}]
Let $\nu^{(N)}:=[\nu^\on]_{\nSn}$. We first consider the case $k=1$.
Let $\var>0$ be given. By Lemma~\ref{lemlegH}, we can find
a bounded continuous function $\varphi$ such that
\[ \int e^{\varphi} g =1; \quad 
\int \varphi\,\d \mu \geq H(\mu|\nu)-\var.\]

On $\nSn$, consider the function
\[ \Phi(v_1,\ldots,v_N) = \varphi(v_1)+\ldots + \varphi(v_N).\]
By Lemma~\ref{lemlegH} again,
\begeq\label{H2t} \frac{H(\mu^{(N)}|\nu^{(N)})}{N} 
\geq \frac1N \int_{\nSn} \Phi(v)\, \d\mu^{(N)}(v)\ -\
\frac1N \log \left(\int e^{\Phi(v)}\d\nu^{(N)}(v)\right).
\endeq

The first term in the right-hand side of~\eqref{H2t} is
controlled by symmetry and convergence of the first marginal:
\begeq\label{H2t1} \frac1N \int_{\nSn} \Phi(v)\, \d\mu^{(N)}(v)
= \int_\R \varphi(v)\, \d(P_1\mu^{(N)})(v)
\ \xrightarrow[N\to\infty]{} \int_\R \varphi\,\d\mu.
\endeq
(Here we have used the continuity of $\varphi$.)

To estimate the second term in the right-hand side of~\eqref{H2t},
we note that
\begeq\label{intePhi} \int e^{\Phi(v)}\d\nu^{(N)}(v)=
\frac{Z_N(e^{\varphi}g,\sqrt{N})}{Z_N(g,\sqrt{N})}
= \frac{Z'_N(e^{\varphi}g,\sqrt{N})}{Z'_N(g,\sqrt{N})}.
\endeq
Since $\varphi$ is bounded above, we know that $e^{\varphi} g$ satisfies
the same estimates as $g$, which makes it possible to apply
Theorem~\ref{thmasymp}: with obvious notation,
\[ Z'_N(e^{\varphi}g,\sqrt{N}) =
\frac{\sqrt{2}}{\Sigma(e^{\varphi} g)}
\Bigl( e^{-\frac{N(1-E(e^{\varphi}g))^2}{2\Sigma(e^{\varphi} g)^2}}
+ o(1) \Bigr)= O (1).\]
Hence
\begeq\label{logZn1} \liminf_{N\to\infty}
\left(-\frac1N \log Z'_N(e^{\varphi}g,\sqrt{N}) \right)\geq 0.
\endeq
Similarly,
\[ Z'_N(g,\sqrt{N}) =
\frac{\sqrt{2}}{\Sigma(g)}
\Bigl( e^{-\frac{N(1-E(g))^2}{2\Sigma(g)^2}}
+ o(1) \Bigr) = \frac{\sqrt{2}}{\Sigma(g)} (1+o(1)),\]
so
\begeq\label{logZn2} \lim_{N\to\infty}
\left(\frac1N \log Z'_N(g,\sqrt{N}) \right)= 0.
\endeq
(This is where the assumption $\int g(x)x^2\,\d x=1$ is used.)

The combination of~\eqref{intePhi}, \eqref{logZn1} and~\eqref{logZn2}
implies
\[ \liminf_{N\to\infty} \left( -\
\frac1N \log \left(\int e^{\Phi(v)}\d\nu^{(N)}(v)\right)\right)
\geq 0.\]
Combining this with~\eqref{H2t1} and~\eqref{H2t}, we find
\[\liminf_{N\to\infty} \frac{H(\mu^{(N)}|\nu^{(N)})}{N} 
\geq \int_\R \varphi\,d\mu,\]
and by the choice of $\varphi$ this is no less than $H(\mu|\nu)-\var$.
Since $\var$ is arbitrarily small, the proof of Theorem~\ref{thmuppergen'}
is complete in the case $k=1$.
\med

Now the proof for general $k$ goes along the same lines: pick up
$\varphi=\varphi(v_1,\ldots,v_k)$ such that
\[ \int_{\R^k} e^{\varphi} g^{\otimes k} =1; \quad 
\int \varphi\,\d \mu_k \geq H(\mu_k|\nu^{\otimes k})-\var.\]

Define $m$ as the integer part of $N/k$.
On $\nSn$, consider the function
\[ \Phi(v_1,\ldots,v_N) = \varphi(v_1,\ldots,v_k)+
\varphi(v_{k+1},\ldots, v_{2k}) +\ldots +
\varphi(v_{(m-1)k+1},\ldots,v_{mk}). \]
Then~\eqref{H2t} is unchanged, and~\eqref{H2t1} transforms into
\begeq\label{H2t1'} \frac1N \int_{\nSn} \Phi(v)\, \d\mu^{(N)}(v)
= \left(\frac{m}{N}\right)\int_{\R^k} \varphi(v)\, \d(P_k\mu^{(N)})(v)
\ \xrightarrow[N\to\infty]{} \frac1k \int_{\R^k} \varphi\,\d\mu_k.
\endeq

Also equation~\eqref{logZn2} is unchanged; but there
is a subtlety with equation~\eqref{logZn1}, which cannot
be directly interpreted in terms of constrained tensor product.
So write $N=km+q$ ($0\leq q\leq k-1$), and
\[ x=(y,z),\qquad y=(x_1,\ldots,x_{km}),\qquad
z=(x_{km+1},\ldots,x_N).\]
By using polar changes of variables (successively for $x$ and $y$) and a test
function argument, we see that
\begin{multline*} \int_{S^{N-1}(\sqrt{N})} F(y,z)\,\d \sigma^N_r(y,z) \ \\
= \
\int_{\R^q} \frac{r}{\sqrt{r^2-|z|^2}} \left(
\int_{S^{km-1}(\sqrt{r^2-|z|^2})} F(y,z)
\:\frac{|S^{km-1}(\sqrt{r^2-|z|^2})|}{|S^{N-1}(r)|}\:\d
\sigma^{km}_{\sqrt{r^2-|z|^2}}\right)\,\d z,
\end{multline*}
where the integral is restricted to the region $|z|\leq r$.
Then we recognize that
\[ \int_{S^{km-1}(\sqrt{r^2-|z|^2})}
e^\Phi(y)\, \left(\frac{g}{\gamma}\right)^{\otimes km}(y)\:
\d\sigma^{km-1}_{\sqrt{r^2-|z|^2}}(y)\
=\ Z'_m\bigl(e^{\varphi}g^{\otimes m},\, \sqrt{r^2-|z|^2}\bigr).\]
So in the end
\begin{multline}\label{repePhi}
\int_{\nSn} e^{\Phi}\, g^\on\, \d\sigma^N \\ =\
\int_{\R^q} \sqrt{\frac{N}{N-|z|^2}}\
\frac{|S^{km-1}(\sqrt{N-|z|^2})}{|S^{N-1}(\sqrt{N})|}\:
\left(\frac{g}{\gamma}\right)^{\otimes q}(z)\
Z'_m\bigl(e^{\varphi} g,\,\sqrt{N-|z|^2}\bigr)\,\d z,
\end{multline}
where the integral is restricted to the region $|z|\leq \sqrt{N}$.
To conclude along the lines of the case $k=1$, it is sufficient to
show that the expression~\eqref{repePhi} is uniformly bounded as
$N\to\infty$ ($\varphi$ and $g$ being fixed). Since there are
a finite number of possible values for $q$, we might also assume
that $q$ is fixed.

Now use the formulas
\[ Z'_m\bigl(e^{\varphi} g,\,\sqrt{N-|z|^2}\bigr)
= \left(\frac{\alpha_{km}(km)}{\alpha_{km}(N)}\right)\times O(1),\]
\[ |S^{N-1}(r)| = \frac{(2\pi)^{N/2}r^{N-1}}{\sqrt{\pi}\alpha_N(N)}
(1+o(1))\]
to estimate~\eqref{repePhi}: after some computations, one finds
\begeq\label{repePhi2} \int_{\nSn} e^{\Phi}\, g^\on\, \d\sigma^N\leq
\frac{O(1)}{(2\pi)^{q/2}} \int_{\R^q} \left(1-
\frac{|z|^2}{N}\right)_+^{\frac{km}2-1} \,
\left(\frac{g}{\gamma}\right)^{\otimes q}\,\d z.\endeq
Next, if $|z|\leq\sqrt{N}$, then
${\displaystyle  \left(1-\frac{|z|^2}{N}\right)^N\leq e^{-|z|^2}}$,
so
\[ \left (1-\frac{|z|^2}{N}\right)_+^{\frac{km}2-1} \leq
e^{-\frac{|z|^2}2 \left(\frac{km-2}{N}\right)}1_{|z|\leq \sqrt{N}}
= e^{-\frac{|z|^2}{2}\left(1- \frac{2+q}N\right)} 1_{|z|\leq\sqrt{N}}.\]
Plug this in~\eqref{repePhi2} to obtain
\begin{align*} \int_{\nSn} e^{\Phi}\, g^\on\, \d\sigma^N & \leq
O(1) \int_{|z|\leq \sqrt{N}} e^{\frac{|z|^2}{N}\left(\frac{2+q}{N}\right)}\,
g^{\otimes q}(z)\,\d z\\
& \leq O(1) e^{1+\frac{q}2}\int_{\R^q} g^{\otimes q}(z)\,\d z
= O(1).
\end{align*}
With this estimate in hand, there is no difficulty to conclude the
proof of Theorem~\ref{thmuppergen'}.
\end{proof}

\section{Convergence of marginals} \label{secmarginals}

This section is devoted to the convergence of finite-dimensional
marginals under various entropy assumptions. In the first subsection,
we show that (with loose notation)
the natural condition $H(\mu^{(N)}|\mu^{\otimes N})=o(N)$
implies that $\mu^{(N)}$ is $\mu$-chaotic. In the second subsection,
we show that at least in the case when $\mu^{(N)}$ is the constrained
tensor product, then the convergence of the finite-dimensional marginals
holds in a stronger sense, namely in relative entropy
(and as a consequence in total variation).

\subsection{From entropy estimates to chaos}

\begin{Thm}[Entropic closeness to the constrained tensor product implies chaos]
\label{thmentrchaos}
Let $\nu(dv)=g(v)\,\d v$ be a probability measure on $\R$, such that
\[ \int g(v)\,v^2\,\d v=1,\qquad \int g(v)\,v^4\,\d v<+\infty,\qquad
g\in L^p(\R)\quad (p>1),\]
and let $\nu^{(N)}=[\nu^{\on}]_{\nSn}$ be the constrained
tensor product of $\nu$ on $\nSn$. Let further $\mu^{(N)}$ be a symmetric
probability measure on $\nSn$, such that
\[ \frac{H(\mu^{(N)}|\nu^{(N)})}{N} \xrightarrow[N\to\infty]{} 0.\]
Then, $\mu^{(N)}$ is $\nu$-chaotic. More precisely, for each $k$,
the marginal $P_k\mu^{(N)}$ converges weakly
(against bounded continuous test functions) to $\nu^{\otimes k}$.
\end{Thm}

\begin{proof}
We first claim that, for given $k$, the sequence $P_k\mu^{(N)}$ is
tight. Indeed, let $m$ be the integer part of $N/k$, then since $N = \int |x|^2\,\d\mu^{(N)}(x) $,
\[ N  \geq
\int (x_1^2+\ldots +x_k^2) \,\d\mu^{(N)}(x) + \ldots +
\int (x_{(m-1)k+1}^2+\ldots+ x_{mk}^2)\,\d\mu^{(N)}(x),\]
and by symmetry the latter expression is
${\displaystyle m \int_{\R^k} |x|^2\,\d(P_k\mu^{(N)})(x)}$.
It follows that
${\displaystyle\int_{\R^k} |x|^2\,\d(P_k\mu^{(N)})(x) \leq \frac{N}{m}}$
which converges to $k$ as $N\to\infty$.

By Prokhorov's theorem, $P_k\mu^{(N)}$ converges, possibly up
to extraction of a subsequence, to some probability measure
$\mu_k$ on $\R^k$. From Theorem~\ref{thmuppergen'},
\[ \frac{H(\mu_k|\nu^{\otimes k})}{k} \leq
\liminf_{N\to\infty} \frac{H(\mu^{(N)}|[\nu^{\otimes N}]_{\nSn})}{N} =0.\]
It follows that $\mu_k=\nu^{\otimes k}$, so the whole sequence
$P_k\mu^{(N)}$ does converge to $\nu^{\otimes k}$, and
$(\mu^{(N)})$ is indeed $\nu$-chaotic.
\end{proof}

\subsection{Marginals of the constrained tensor product}

As an obvious consequence of Theorems~\ref{thmuppergen'} and ~\ref{thmentrchaos}, the constrained
tensor product $\mu^{(N)}$ of $\mu$ is itself $\mu$-chaotic.
We shall show in this section a stronger result:
$P_k\mu^{(N)}$ converges to $\mu^{\otimes k}$ in total variation,
and even in relative entropy.

\begin{Thm}[Property $(ii^\prime)$ for the constrained tensor product]
\label{ai} 
Let $\mu(\d v)= f(v)\,d v$ be a probability measure on $\R$, such that
$f\in L^p(\R)$ for some $p>1$, and $\int_\R v^4f(v)\d v < \infty$.
Let $\mu^{(N)} = [\mu^{\otimes N}]_{S^{N-1}(\sqrt{N})}$ be
the restricted $N$-fold tensor product of $\mu$.
Then for all positive integers $k$,
$$\lim_{N\to\infty} H\bigl(P_k\mu^{(N)} | \mu^{\otimes k}\bigr) = 0\ .$$
\end{Thm}

\begin{proof} Let $[f^\on]_{\nSn}$ stand for the density of
the constrained tensor product, with respect to the uniform probability
measure $\sigma^N$. Fix any integer $k$. Then for all $N$ sufficiently large, 
$$[f^{\otimes N}]_{S^{N-1}(\sqrt{N})} = 
\left(\prod_{j=1}^k(f(v_j)/\gamma(v_j))\right)\frac{
\prod_{j=k+1}^N(f(v_j)/\gamma(v_j))}{Z'_N(f,\sqrt{N})}$$
With the notation $s^2=\sum_{j=1}^k v_j^2$, this expression
can be rewritten as
$$
[f^{\otimes N}]_{S^{N-1}(\sqrt{N})} = 
\left(\prod_{j=1}^k(f(v_j)/\gamma(v_j))\right)
\frac{ Z'_{N-k} (f,\sqrt{N-s^2})} {Z'_N(f,\sqrt{N})}\
\frac{\prod_{j=k+1}^N(f(v_j)/\gamma(v_j))} {Z'_{N-k}(f,\sqrt{N-s^2})}\ .$$
Therefore,
${\displaystyle P_k\left(\mu^{(N)}\right) = 
\left(\prod_{j=1}^k(f(v_j)/\gamma(v_j))\right)
\frac{ Z'_{N-k} (f,\sqrt{N-s^2})} {Z'_N(f,\sqrt{N})}\ P_k(\sigma^N)}$.

As a consequence, with ${\cal L}_k$ standing for the
$k$-dimensional Lebesgue measure,
\begin{align*}
H(P_k\mu^{(N)}|\mu^{\otimes k}) &= 
\int_{\R^k}\bigl(\log \frac{\d P_k(\mu^{(N)})}{\d {\cal L}_k} - 
\log f^{\otimes k}\bigr)\, \d (P_k\mu^{(N)})\\
&= \int_{\R^k}  \log\left(\frac{ Z'_{N-k} (f,\sqrt{N-s^2})} 
{Z'_{N-k}(f,\sqrt{N})}\right)\,\d(P_k\mu^{(N)})\\
& \qquad 
+  \int_{\R^k}\log\frac{\d (P_k\sigma^N)}{\d \gamma^{\otimes k}}
\,\d(P_k\mu^{(N)})\\
\end{align*}
{From} Theorem~\ref{thmasymp} and some computation,
$$\log\left(\frac{ Z'_{N-k} (f,\sqrt{N-s^2})} {Z'_{N-k}(f,\sqrt{N})}\right) 
= \left( e^{- s^4/(2\theta N)}\right)(1+o(1))$$
and so
$$H(P_k \mu^{(N)} |f^{\otimes k}) = 
\int_{\R^k}\log\frac{\d (P_k\sigma^N)}{\d \gamma^{\otimes k}}
\,\d(P_k\mu^{(N)}) + o(1).$$
In other words,
\begeq\label{otherHPk}
H(P_k\mu^{(N)}|\mu^{\otimes k}) = \int \Psi_N(y)\,f^{\otimes k}(y)\,
\d y \ + o(1),
\endeq
where
\begeq\label{PsiN} \Psi_N(y):=
\left( \frac{P_k(\sigma^N) } {\gamma^{\otimes k}}\right) \log 
 \left(\frac{P_k(\sigma^N) }{\gamma^{\otimes k}}\right) 
  \left(\frac{ Z'_{N-k} (f,\sqrt{N-|y|^2})} {Z'_{N-k}(f,\sqrt{N})}\right).
\endeq

Now let us derive some estimates on $\Psi_N$.
By direct computation,
\begeq (P_k\sigma^N)(\d y) = \frac{|S^{N-k-1}|}{N^{k/2}|S^{N-1}|}
\left(1 - \frac{|y|^2}{N}\right)^{(N-k-2)/2}\,\d y.
\endeq
By an application of Stirling's formula, and some computation again,
\begeq\label{csqSt} \frac{\d(P_k\sigma^N)}{\d \gamma^{\otimes k}}(y)
\leq (1+o(1))e^{\frac{(k+1)|y|^2}{N}}1_{|y|\leq \sqrt{N}}
\leq (1+o(1)) e^{k+1} = O(1).
\endeq
On the other hand, Theorem~\ref{thmasymp} implies that
the ratio of the $Z'$ terms in~\eqref{PsiN} is uniformly bounded.
We conclude that $\Psi_N(y)$ itself is bounded above, uniformly
in $N$ and $y$. This makes it possible to apply the dominated 
convergence theorem, in the form
\[ \limsup_{N\to\infty} \int \Psi_N(y)\, f^{\otimes k}(y)\,\d y
\leq \int \left(\limsup_{N\to\infty}\Psi_N(y)\right)\,
f^{\otimes k}(y)\,\d y.\]

It follows from~\eqref{csqSt} and Theorem~\ref{thmasymp} that
for any $y\in\R^k$, 
\[\begin{cases}
\dps \frac{\d P_k\sigma^N}{\d \gamma^{\otimes k}}(y)\longrightarrow 1\\ \\
\dps \frac{ Z'_{N-k} (f,\sqrt{N-|y|^2})} {Z'_{N-k}(f,\sqrt{N})}
\longrightarrow 1\end{cases}\]
as $N\to\infty$. So $\lim \Psi_N(y)=0$, and as a consequence
${\displaystyle  \limsup_{N\to\infty} \int\Psi_N(y) \, f^{\otimes k}(y)\,\d y\leq 0}$,
so by~\eqref{otherHPk},
${\displaystyle \limsup_{N\to\infty} H(P_k\mu^{(N)}| \mu^{\otimes k}) \leq 0}$.
This concludes the proof of Theorem~\ref{ai}.
\end{proof}

\section{From microscopic to macroscopic entropy} \label{secmicromacro}

Now comes one of the main results of this paper.

\begin{Thm}\label{thmmicromacro}
Let $f$ be a probability density on $\R$, such that
\[ \int f(v)\,v^2\,\d v=1\qquad
\int f(v)\,v^4\,\d v <+\infty\qquad
f\in L^\infty(\R). \]
Let $\nu(\d v)=f(v)\,d v$, and let
$\nu^{(N)}=[\nu^\on]_{\nSn}$ be the constrained $N$-fold
tensor product of $\nu$. For each $N$, let further
$\mu^{(N)}$ be a probability density on $\nSn$ such that
\[ \frac{H(\mu^{(N)}|\nu^{(N)})}{N}\xrightarrow[N\to\infty]{} 0.\]
Then
\[ \frac{H(\mu^{(N)}|\sigma^N)}{N}\xrightarrow[N\to\infty]{}
H(\nu|\gamma).\]
\end{Thm}

\begin{Rk} \label{rkwL8} During the proof, we shall show that the convergence
of the marginals $P_k\mu^{(N)}$ actually holds true in the sense
of weak convergence against bounded measurable functions
(as opposed to bounded continuous functions). We do not know whether
it holds true in the sense of, say, total variation.
\end{Rk}

\begin{proof}[Proof of Theorem~\ref{thmmicromacro}]
First we write
\begin{align*} H(\mu^{(N)}|\sigma^N) & = 
\int \log\frac{\d\mu^{(N)}}{\d\sigma^N}\,d\mu^{(N)} \\
& = \int \log\frac{\d\mu^{(N)}}{\d\nu^{(N)}}\,\d\mu^{(N)}
\ + \ \int \log\frac{\d\nu^{(N)}}{\d\sigma^N}\,\d\mu^{(N)}\\
& = H(\mu^{(N)}|\nu^{(N)}) + 
\int \log\left(\frac{f}{\gamma}\right)^{\on}\,\d\mu^{(N)}
- \log (Z'_N) \\
& = o(N) + N \int \log f(v_1)\,\d\mu^{(N)}(v) -
\int \log \gamma^{\on}\,\d\mu^{(N)} - \log Z'_N(f,\sqrt{N})\\
& = o(N) + N \int \log f(v_1)\,\d\mu^{(N)}(v)
+ N \left(\frac{1+\log (2\pi)}{2}\right)\ .
\end{align*}
In the next to last step, we have used  Theorem~\ref{thmasymp}, which implies that
$\log Z'_N(f,\sqrt{N})$ converges to a positive limit as $N\to\infty$, and so may be absorbed into the $o(N)$ term.
Also, in the last step we have replaced $\gamma^{\on}$ by its explicit
expression on $\nSn$.

Then, after division by $N$, we find
\begin{align*}
\frac{H(\mu^{(N)}|\sigma^N)}N & = \int \log f(v_1)\,\d\mu^{(N)}(v)
+ \left(\frac{1+\log (2\pi)}{2}\right) + o(1)\\
& = \int_{\R} \log f(v_1)\,\d(P_1\mu^{(N)}) (v_1)
+ \left(\frac{1+\log (2\pi)}{2}\right) + o(1).
\end{align*}

For any $\delta>0$, we have therefore
\begeq\label{step1'}
\frac{H(\mu^{(N)}|\sigma^N)}N \leq
\int \log (f(v_1)+\delta)\,\d(P_1\mu^{(N)})(v_1)
+ \left(\frac{1+\log (2\pi)}{2}\right) + o(1).
\endeq

Assume for the moment that $f$ is continuous. Then $\log (f+\delta)$ 
is a bounded continuous function, so we can pass to the limit, using the
weak convergence of $P_1\mu^{(N)}$ to $\nu(\d v)=f(v)\,d v$
(Theorem~\ref{thmentrchaos}), and deduce
\begeq\label{step1''}
\limsup_{N\to\infty} \frac{H(\mu^{(N)}|\sigma^N)}N \leq
\int f(v_1) \log (f(v_1)+\delta)\,\d v_1
+ \left(\frac{1+\log (2\pi)}{2}\right).
\endeq
By dominated convergence, we can now let $\delta\to 0$, and recover
\[ \limsup_{N\to\infty} \frac{H(\mu^{(N)}|\sigma^N)}N \leq
\int f(v_1) \log f(v_1)\,\d v_1
+ \left(\frac{1+\log (2\pi)}{2}\right).\]
Since $\int f(v)v^2\,\d v=1$, it is easy to check that the latter
expression coincides with $H(\nu|\gamma)$. The conclusion is that
\begeq\label{limsupentr1}
\limsup_{N\to\infty} \frac{H(\nu^{(N)}|\sigma^N)}N \leq
H(\nu|\gamma).
\endeq
On the other hand, by Theorem~\ref{thmuppergen'}, applied with
$g=\gamma$,
\begeq\label{limsupentr2}
H(\mu|\gamma) \leq \liminf_{N\to\infty} \frac{H(\mu^{(N)}|\sigma^N)}N.
\endeq
The combination of~\eqref{limsupentr1} and~\eqref{limsupentr2} concludes
the proof of Theorem~\ref{thmmicromacro}.
\end{proof}
\med

Now, let us prove the more general statement alluded to in Remark~\ref{rkwL8}.
We start again from~\eqref{step1'}, and deduce that
\[ \frac{H(\mu^{(N)}|\sigma^N)}N \leq \log \|f\|_{L^\infty}
+ \left(\frac{1+\log (2\pi)}{2}\right) + o(1),\]
which is bounded as $N\to\infty$.
This bound can be combined with the exact (not asymptotic) inequality
\begin{equation}\label{more}
 \frac{H(P_k\mu^{(N)}|P_k\sigma^N)}{k} \leq
2\ \frac{H(\mu^{(N)}|\sigma^N)}{N},
\end{equation} to obtain
\[ H(P_k\mu^{(N)}|P_k\sigma^N) = O(1).\]

The inequality (\ref{more}) is a generalization of the subadditivity inequality on $S^N$ from
\cite{CLL:entropy}, which gives the $k=1$ case. The generalization to higher $k$ can be found in
\cite{BCM}, in Example 1 under Corollary 5 there. 

Next, by the same kind of computation as in the beginning of the proof,
\[ H(P_k\mu^{(N)}|\gamma^{\otimes k}) =
H(P_k\mu^{(N)}|P_k\sigma^N) + \int 
\log\frac{\d(P_k\sigma^N)}{\d \gamma^{\otimes k}}\,\d(P_k\mu^{(N)}).\]
It follows by~\eqref{csqSt} that
\[ H(P_k\mu^{(N)}|\gamma^{\otimes k})  \leq H(P_k\mu^{(N)}|P_k\sigma^N)
+ C,\]
where $C$ is some constant depending only on $k$. In particular,
$H(P_k\mu^{(N)}|\gamma^{\otimes k})$ is bounded as $N\to\infty$.
The conclusion is that {\em the marginals $P_k\mu^{(N)}$ have bounded
relative entropy with respect to $\gamma^{\otimes k}$,
uniformly in $N$.} It follows by the Dunford-Pettis compactness criterion
that the densities $f_k^{(N)}$ of
$P_k\mu^{(N)}$ constitute a compact set in $L^1(\R^k)$, equipped
with the weak topology. Since this family converges weakly to
$f^{\otimes k}$ as $N\to\infty$, actually the limit
\[ \int_{\R^k} \psi(v) f_k^{(N)}(v)\,\d v \xrightarrow[N\to\infty]{}
\int_{\R^k} \psi(v) f^{\otimes k}(v)\,\d v\]
holds true for all bounded measurable function $\psi$, not necessarily
continuous. The conclusion follows by the same arguments as before.

\section{Generalization to unbounded densities} \label{secunbounded}

In this section we use a density argument to derive Theorem~\ref{thmexist}
from Theorem~\ref{thmmicromacro}.

\medskip
\begin{proof}[Proof of Theorem~\ref{thmexist}]  
If $f$ is bounded and has a finite fourth moment, we can simply use
the tensor product construction.
Otherwise, we define approximations to $f$ as follows: If $f$ has a finite fourth moment but is unbounded, and $\delta>0$, define
$f_\delta$ to be $e^{\delta \Delta}f$, rescaled so that $f_\delta$ has unit variance.  Otherwise, if
 $f$ does not have a finite fourth moment, let 
 $$g_\delta = e^{\delta \Delta}\bigl(f1_{[-1/\delta,1/\delta]}\bigr)\ .$$
 Then renormalize $g_\delta$ so that it is a probability density, and finally, make an affine change of variable
to obtain a density that has zero mean and unit variance. Call this $f_\delta$.

It is easy to see that for any positive integer $j$, we can choose a value $\delta_j > 0$ so that
\begeq\label{fir}
\bigl| H(f_{\delta_j} | \gamma) - H(f|\gamma)\bigr| < \frac{1}{2j}\ .
\endeq
Apply the tensor product construction with each $f_{\delta_j}$ to produce the chaotic sequence
$\mu^{(N)}_{\delta_j}$.  By Theorem~\ref{thmmicromacro},
$$\frac{1}{N} H(\mu^{(N)}_{\delta_j} | \sigma^N) \xrightarrow[N\to\infty]{}
H(f_{\delta_j} | \gamma)\ .$$
Therefore, we may inductively define an increasing sequence of 
integers $\{N_j\}$  by choosing $N_j > N_{j-1}$ large enough that
\begeq\label{sec}
\Bigl| \frac{1}{N} H(\mu^{(N)}_{\delta_j} | \sigma^N) - 
H(f_{\delta_j} | \gamma)\Bigr | < \frac{1}{2j}
\endeq
for all $N > N_j$.

Combining~\eqref{fir} and~\eqref{sec}), we obtain
\begeq\label{ter}
\Bigl| \frac{1}{N} H(\mu^{(N)}_{\delta_j} | \sigma^N) 
- H(f | \gamma)\Bigr| < \frac{1}{j}\ .
\endeq

Further increasing the $N_j$ if required, we may assume, 
on account of Theorem~\ref{ai}, that for each $j$,
\begeq\label{for}
N\geq N_j\Longrightarrow\qquad
\sup_{1\leq \ell\leq j} H( P_\ell \mu_{\delta_j}^{(N)} | 
f_{\delta_j}^{\otimes \ell}) < \frac{1}{2\,j^2}\ .
\endeq

We are now ready to define our sequence, which we shall show to be $f(v){\rm d}v$--chaotic in the entropic sense:  For each $N$, define
$$\mu^{(N)} = \mu^{(N_k)}_{\delta_k}\qquad{\rm for}\qquad  k = \inf\{\ell\ :\ N_\ell < N\}\ .$$

First, property $(i)$ holds for obvious reasons. Next, to see that property $(ii)$ holds, let 
$\phi$ be any continuous bounded function on $\R^k$. Then, by the 
well--known Csiszar--Kullback--Leibler--Pinsker inequality 
and~\eqref{for},
\[ \|P_k \mu^{(N)} - f_{\delta_k}^{\otimes k}\|_{L^1(\R^k)}
\leq \sqrt{2 H(P_k\mu_{\delta_k}^{(N)}|f_{\delta_k}^{\otimes k})}
< \frac1{k} \]
for all $N > N_k$.
Therefore, for all $N>N_k$,
$$\left| \int_{\R^k}\phi\, \d P_k \mu^{(N)}  - 
\int_{\R^k}\phi  f_{\delta_k}^{\otimes k}\d v\right| < 
\frac{\|\phi\|_\infty}{k}\ ,$$
while trivial estimates show that
$$\lim_{\delta\to 0}\left|
\int_{\R^k}\phi  f_{\delta_k}^{\otimes k}\d v - 
\int_{\R^k}\phi  f^{\otimes k}\d v\right|  = 0\ .$$
 
Finally, the fact that $(iii)$ holds follows easily from  (\ref{ter}).

\end{proof}

\section{Entropy production bounds} \label{production}

In this section, we prove Theorem 3 We first construct initial data $f$ for the Boltzmmann--Kac
equations that has low entropy production. 
We then show that this implies that  the $f \d v$--chaotic family of initial data for the Kac master equation also has low entropy production.

Given two probability densities $f$ and $g$  on the line $\R$, define
$$(f\circ g)(v)= \frac{1}{2\pi}\int_{\R}\int_0^{2\pi}f(\cos(\theta)v-\sin(\theta)v_*)
\ g(\sin(\theta)v+\cos(\theta)v_*){\rm d}\theta{\rm d}v_*\ .$$
The density $f\circ g$ is called the {\em Wild convolution} of $f$ and $g$, and using it we may write the
Boltzmann--Kac equation in the compact form
\begin{equation}\label{compact}
\frac{\partial}{\partial t}f = f\circ f - f\ .
\end{equation}

Any rescaling of $\gamma(v)$ is an equilibrium solution of this equation: 
For any $a>0$, define 
$$M_a = (2\pi a)^{-1/2}\exp(-v^2/2a) = a^{-1/2}\gamma(a^{-1/2} v)\ .$$
There are the so--called {\em Maxwellian densities}, and one easily sees that for all $a>0$, 
$M_a\circ M_a = M_a$, so that these are stationary solutions of (\ref{compact}).

For any zero mean, unit variance solution $f$ of (\ref{compact}), 
the relative entropy with respect to $\gamma$ satisfies
\begin{equation}\label{pro1}
-\frac{{\rm d}}{{\rm d}t}H(f|\gamma) =  \int_{\R}(-\ln f)[f\circ f - f]{\rm d}v\ .
\end{equation}
The analog of the Boltzmann $H$--Theorem for the Boltzmann--Kac equation asserts that this
quantity is strictly positive unless $f$ is one of the Maxwellians, in which case it is zero.

Our first goal in this section is to construct, for each $c>0$, 
zero mean, unit variance initial data $f$ for which 
$$ \int_{\R}(-\ln f)[f\circ f - f]{\rm d}v  < c \ H(f|\gamma)\ .$$

There is a very natural construction that has been exploited by Bobylev and Cercignani \cite{BoCer}
 in the case of
the actual Boltzmann equation:  {\em Use a superposition of two very different Maxwellians}. 
This is natural since each $M_a$ is  an equilibrium solution. (In the case of the actual Boltzmann equation
there is an even larger class of equilibrium densities to work with since momentum is also conserved. Here,
only centered Maxwellians are equilibrium solutions.)

 Pick a small positive number $\delta$, and define
\begin{equation}\label{fedefin}
f = (1-\delta)M_a + \delta M_b
\end{equation}
where 
\begin{equation}\label{fedefin2}
b = 1/(2\delta)\qquad{\rm and}\qquad a = 1/(2(1-\delta))\ .
\end{equation}
 Then since
${\displaystyle \int_\R M_a{\rm d}v = 1}$ and ${\displaystyle\int_\R v^2M_a{\rm d}v = a}$, 
$$\int_{\R} v^2 (1-\delta)M_a {\rm d}v = 
\int_{\R} v^2 \delta M_b {\rm d}v = \frac{1}{2}\ ,$$
so that each Maxwellian component contributes half of the energy, though for small $\delta$,
most of the mass is contained in the $M_a$ component.

\begin{Prop} \label{propDsmall}
For any $c>0$ there is a probability distribution $f$ on $\R$ such that
$\int v^2 f(v){\rm d}v = 1$ and
\[ \frac{D(f)}{H(f|\gamma)} \leq c,\]
where $D(f) = \int_\R (-\ln f)\, [ f\circ f - f] {\rm d}v$ is the entropy
production for the Boltzmann--Kac equation. Moreover, $f$ can be chosen
smooth with finite moments of all orders, in fact a linear combination of
Gaussian functions.
\end{Prop}

\begin{proof}[Proof of Proposition~\ref{propDsmall}]
Define
$
f = (1-\delta)M_a + \delta M_b
$ as above.
We first show  that
$H(f|\gamma)$ is bounded away from $0$ uniformly for all $\delta$
sufficiently small. In fact,
\begin{equation}\label{adelta}
\lim_{\delta\to 0}H(f|\gamma) = \frac {\ln{2}}{2}\ .
\end{equation}

To prove \eqref{adelta}, we use the definition 
of $f$ and  the monotonicity of the logarithm:
\begin{eqnarray*}
 \int_{\R} f\ln\left(\frac{f}{\gamma}\right){\rm d}v &=&
(1-\delta) \int_{\R} M_a \ln\left(\frac{f}{\gamma}\right){\rm d}v + \delta
 \int_{\R} M_b \ln\left(\frac{f}{\gamma}\right){\rm d}v\nonumber\\
&\ge&
(1-\delta) \int_{\R} M_a \ln\left(\frac{(1-\delta)M_a}{\gamma}\right){\rm d}v + \delta
 \int_{\R} M_b \ln\left(\frac{\delta M_b}{\gamma}\right){\rm d}v\nonumber\\
&=& (1-\delta)\ln(1-\delta) + (1-\delta)H(M_a|\gamma) + \delta\ln\delta + \delta H(M_b|\gamma)\ \nonumber\\
&=& (1-\delta)\ln(1-\delta) + \delta\ln\delta + \frac12
\bigl(\ln 2 - \delta + \delta\ln (2\delta)\bigr)\ ,
\end{eqnarray*}
where the last equality follows from the formula
${\displaystyle H(M_c|\gamma) = \frac{1}{2}(c-1) - \frac{1}{2}\ln c}$.
This implies \eqref{adelta} at once.

It remains to estimate the entropy production associated with $f$.
First, $$f\circ f = (1-\delta)^2M_a + \delta^2 M_b + 2\delta(1-\delta)M_a\circ
M_b\ .$$ Therefore,
$$
\int_{\R}(-\ln f)[f\circ f - f]{\rm d}v  =
(\delta-\delta^2)\int_{\R} (-\ln f)[ 
2M_a\circ M_b - (M_a +  M_b)]{\rm d}v\ .
$$
Next, use the fact that for all $\delta<1$, $(-\ln f) \ge 0$.  Hence,
we can simplify, obtaining
$$\int_{\R}(-\ln f)[f\circ f - f]{\rm d}v \le 
2\delta  \int_{\R}(-\ln f)M_a\circ M_b{\rm d}v\ .$$

By monotonicity of the logarithm,
${\displaystyle \ln f \ge \ln(\delta M_b)}$
so that
$$(-\ln f) \le -\frac{1}{2}(\ln \delta - \ln\pi) + \delta v^2\ .$$
Hence we have
${\displaystyle
\int_{\R}(-\ln f)[f\circ f - f]{\rm d}v \le 2\delta\left(
-\frac{1}{2}(\ln \delta - \ln\pi)\right) + 2\delta^2}$.

Evidently, the leading term is 
$-\delta \ln \delta$. So
\begin{equation}\label{bdelta} 
\lim_{\delta\to 0}\int_{\R}(-\ln f)[f\circ f - f]{\rm d}v =0\ . 
\end{equation}

The combination of~\eqref{adelta} and \eqref{bdelta} implies
Proposition~\ref{propDsmall} at once.
\end{proof} 

\begin{Rk} The distribution $f$ constructed above do not have
uniformly bounded fourth moment. In fact, 
the ratio of $\int v^4f{\rm d}v$ to $\int v^4 M{\rm d}v$ tends to infinity
like $1/\delta$.
\end{Rk}

Now we shall deduce Theorem~\ref{thmslow} from Proposition~\ref{propDsmall}:

\begin{proof}[Proof of Theorem~\ref{thmslow}]
Let $c>0$, and let $f$ be defined by Proposition~\ref{propDsmall}.
Let further
${\displaystyle F^N = [f^{\otimes N}]_{\nSn}}$.
Theorem~\ref{thmsuffschaos} guarantees that
${\displaystyle \frac{H(F^N)}{N} \longrightarrow H(f|\gamma)}$.
Thus,  it suffices to establish
$$\lim_{N\to \infty} \frac{1}{N} \langle \ln(F^N), L_N(F^N) \rangle_{L^2(\nSn)} = 2\int_{\R} (-\ln f)[f\circ f - f]{\rm d} v\ .\label{lnff}$$

To prove~\eqref{lnff}, we first write
${\displaystyle\log F^N  = \ln Z_N(f,\sqrt{N}) + \sum_{k=1}^N\log f(v_k)}$,
and hence $-\langle L_N(F^N), \ln(F^N)\rangle_{L^2(\nSn)}$ is given by
\begin{eqnarray}\label{prod4}
&\phantom{=}&\int_{\nSn}\left[ -\ln(Z_N(f,\sqrt{N}) + \sum_{k=1}^N\log f(v_k)\right]N\left(QF^N - F^N\right)\d \sigma^N\nonumber\\
&=&\int_{\nSn}\left[ \sum_{k=1}^N\log f(v_k)\right]N\left(QF^N - F^N\right)\d \sigma^N\nonumber\\
&=& \sum_{k=1}^N \sum_{i<j}^N\int_{\nSn}\log f(v_k)\frac{2}{N-1}\nonumber\\
&\times&
\left(\frac{1}{2\pi}\int_0^{2\pi}\left( f^{\otimes N}(R_{i,j,\theta}V)- f^{\otimes N}(V)\right)
 \d \theta\right) \frac{1}{Z_N(f,\sqrt{N})}\d \sigma^N\ .\nonumber\\
\end{eqnarray}

By the invariance  of $\sigma^N$ under rotations, the integral over $\nSn$ vanished unless
$k =i$ or $k=j$. By the permutation symmetry, we may set $k=1$, and account for the sum over $k$
by multiplying by $N$. 
There are $N-1$ pairs of which $1$ is a member. We may set $i,j = 1,2$, and account for the sum over pairs by
multiplying by $N-1$. 
We are then left with
\begin{multline}2N \int_{\nSn}\log f(v_k)
\left[\frac{1}{2\pi}\int_0^{2\pi}\left( f(v_1')f(v_2') - f(v_1)f(v_2)\right)
 \d \theta\right]\\ 
 \times \ \frac{1}{Z_N(f,\sqrt{N})}\prod_{j=3}^N f(v_j) \d \sigma^N
 \end{multline}
 where
 ${\displaystyle v_1' = \cos(\theta)v_1 + \sin(\theta)v_2}$ and 
 ${\displaystyle v_2' = -\sin(\theta)v_1 + \cos(\theta)v_2}$.
 
 Since $f$ is a linear combination of Maxwellian densities, there is a constant $C$
 such that $|\log f(v)| \le C(1+v^2)$. Then again since  $f$ is a linear 
 combination of Maxwellian densities,
 ${\displaystyle\Bigl| \log f(v) \left( f(v_1')f(v_2') - f(v_1)f(v_2)\right) \Bigr|}$
 is bounded, and in fact has Gaussian decay. 

The proof will be completed by showing that
$$\lim_{N\to\infty}P_2\left(\frac{1}{Z_N(f,\sqrt{N})}\prod_{j=3}^N f(v_j) \d \sigma^N\right) = \d v_1\d v_2\ .$$
Since, with $s^2 = v_1^2+v_2^2$, 
$$P_2\left(\frac{1}{Z_N(f,\sqrt{N})}\prod_{j=3}^N f(v_j) \d \sigma^N\right) = 
\frac{Z_{N-2}(f,\sqrt{N -s^2})}{Z_N(f,\sqrt{N})}P_2(\d \sigma_N)\ ,$$
it remains only to show that
\begin{equation}\label{prop5}
\lim_{N\to\infty} \frac{Z_{N-2}(f,\sqrt{N -s^2})}{Z_N(f,\sqrt{N})} = 2\pi e^{(v_1^2 + v_2^2)/2}
\end{equation}
since it is well known, and easily follows from (\ref{expligamma}), that
$$\lim_{N\to\infty} P_2(\d \sigma^N) = \frac{1}{2\pi} e^{-(v_1^2 + v_2^2)/2}\d v_1\d v_2\ .$$
However, (\ref{prop5}) easily follows from Theorem~\ref{propct} and Stirling's formula.
\end{proof}

\begin{Rk} We have taken advantage of Maxwellian bounds on $f$ to shorten the proof, but
a similar result could be obtained in the same way for more general densities $f$ by arguing as in
the proof of Theorem~\ref{ai}. A more challenging problem would be to prove an analog of
Lemma~\ref{lnff} for a more general class of chaotic data than conditioned tensor products.
\end{Rk}

\appendix

\section{ An entropic Local Central Limit Theorem} \label{secLCLT}

The terminology ``Local Central Limit Theorem'' is used 
to designate a version of the Central Limit Theorem in which the conclusion
is strengthened from weak convergence of the law to locally uniform
pointwise convergence of the densities~\cite{feller:vol2:71}.  

As recalled in the introduction,  such a theorem can only hold if the common
law of the independent random variables has a density $f$ that satisfies 
certain regularity hypotheses --- in particular, $f$ is usually
require to belong to $L^p$ for some $p>1$. 

Of course the rate of pointwise convergence depends on the 
regularity of $f$. In this paper, we require precise, quantitative 
information on the rate, and the version of the Local Central Limit 
Theorem that we prove here provides this.

There is a remarkable feature that emerges:
When an $L^p$ bound is imposed on $f$, then the asymptotic rate of
(pointwise) convergence of the densities $\sqrt{N}f^{\ast N}(\sqrt{N}x)$
to the Gaussian distribution can be estimated in terms of
only the relative entropy $H(f|\gamma)$,
even if the assumption $H(f|\gamma)<\infty$ alone is {\it not}
sufficient to ensure the convergence!
Of course, the $L^p$ bound on $f$ enters the estimates of convergence,
{\em but only in determining how large $N$ must be before the universal 
rate estimates governed by $H(f|\gamma)$ become valid}.
For this reason, we refer to the result obtained here as an 
Entropic Local Central Limit Theorem.

In addition to the $L^p$ bound on $f$, one requires a certain measure of localization of $f(v)v^2{\rm d}v$,
which is usually taken care of in the assumptions by assuming that
\begeq\label{extrae}
\int_\R f(v)v^{2+\epsilon}{\rm d}v < \infty
\endeq
for some $\epsilon > 0$.

So that we may include all finite energy initial data in our conclusions, we wish to avoid such an assumption.  For this reason, we introduce the function
\[ \chi(r) = \int_{|x|\geq \frac1r} |x|^2 f(x)\,dx\ .\]
This is a continuous function vanishing at $r=0$, and the rate at which it vanishes as $r\to 0$
gives the sort of control that would be provided by (\ref{extrae}). 
Indeed, under the assumption (\ref{extrae}),
\[\chi(r) = {\cal O}(r^\epsilon)\ .\]

Before beginning the proof, we note that
in this Appendix the state space is $\R^k$, where $k$ is some positive
integer. All the constants in our results may depend on $k$, but this
dependence will not be explicitly recalled.

We start by some properties of the Fourier transform of probability
densities. In the sequel, we use the following convention for the Fourier
transform in $\R^k$:
\[ \hat{f}(\xi) = \int_{\R^k} e^{-2i\pi x\xi} f(x)\,dx.\]

\begin{Prop}\label{propFourier}
Let $g$ be a probability density on $\R^k$, such that
\[ \int_{\R^k} x g(x)\,\d x\ =\ 0,\qquad 
\int_{\R^k} (x\otimes x)\, g(x)\, \d x\ =\ I_k,\qquad
\int_{\R^k} g \log g \d x \: \leq \: H <+\infty,\]
where $I_k$ is the $k\times k$ identity matrix.
Let further $\chi$ be such that
\[ \int_{|x|\geq \frac1r} |x|^2 g(x)\,dx\leq \chi(r),\]
where $\chi(r)$ goes to~0 as $r\to 0$.
Then
\sm

(i) Given $\eta>0$ there is $\alpha=\alpha(H,\eta)>0$ such that
\[ |\xi|\geq \eta \Longrightarrow\quad
|\hat{g}(\xi)|\leq 1-\alpha.\]
\sm

(ii) There is a function $\var(\delta)=\var(H,\chi,\delta)$, going to~0
as $\delta\to 0$, such that
\[ |\xi|\leq\delta\Longrightarrow\quad  \Bigl| \hat{g}(\xi) -
\bigl( 1-2\pi^2|\xi|^2\bigr) \Bigr| \leq \var(\delta) |\xi|^2.\]
\sm

(iii) There is $\alpha_0=\alpha_0(H,\chi)$ such that
\[ \forall\xi\in\R^k\qquad
|\hat{g}(\xi)| \leq \max \bigl( 1-\pi^2|\xi|^2,\, 1-\alpha_0\bigr).\]

\end{Prop}

\begin{proof} First, it is clear that $(iii)$ follows from $(i)$ and $(ii)$.
For simplicity, we shall prove $(i)$ and $(ii)$ only in the case $k=1$;
the generalization to higher dimension does not bring in any major complication.\med

Let us prove $(i)$. Let $\xi$ be such that
$|\xi|\geq \eta$, and let $z$ be such that $\hat{g}(\xi) e^{-2i\pi z\xi} =
|\hat{g}(\xi)|$. Then, with $\Re$ standing for real part,
\begin{align*}
|\hat{g}(\xi)| & = \Re \bigl[ \hat{g}(\xi) e^{-2i\pi z\xi}\bigr]\\
& = \Re \left( \int g(x) e^{-2i\pi(x+z)\xi}\,\d x\right)\\
& = \int g(x) \cos \bigl[ 2\pi(x+z)\xi\bigr]\,\d x.
\end{align*}

Let $R>1$, to be chosen later. Write
\begin{align*} |\hat{g}(\xi)| & = 
\int g(x)\d x\ - \ \int g(x) \Bigl( 1-\cos \bigl[2\pi(x+z)\xi]\Bigr)\,\d x\\
& = 1 \ - \ \int g(x) \Bigl( 1-\cos \bigl[2\pi(x+z)\xi]\Bigr)\,\d x\\
& \leq 1 \ - \ 
\int_{[-R,R]} g(x) \Bigl( 1-\cos \bigl[2\pi(x+z)\xi]\Bigr)\,\d x.
\end{align*}
So it is sufficient to show that
\[ \int_{[-R,R]} g(x) \Bigl( 1-\cos \bigl[2\pi(x+z)\xi]\Bigr)\,\d x
\ \geq \ \alpha>0.\]

Let $\beta\in (0,1/2)$, to be chosen later; define
\[ B:= \Bigl\{ x\in [-R,R];\ 1-\cos\bigl[2\pi(x+z)\xi]\leq \beta\Bigr\}.\]
The point is to show that $\int_B g$ is small when $\beta$ is small too.
For this we shall show that $B$ has small Lebesgue measure and use the
entropy bound on $g$.

If $x$ lies in $B$, then $|x|\leq R$, and there exists $n\in\Z$ such that
\[ \left| x-\frac{n}{\xi}\right| \leq \frac{\cos^{-1}(1-\beta)}{2\pi|\xi|}.\]
So $B$ consists of at most $2R|\xi|+3$ intervals, with width
$\cos^{-1}(1-\beta)/(\pi|\xi|)$, which can be bounded by
$\sqrt{2\beta}/(\pi|\xi|)$. So the Lebesgue measure $|B|$
of $B$ can be estimated as follows:
\begin{align} 
|B| \leq \left(\frac{2R|\xi|+3}{\pi|\xi|}\right) \sqrt{2\beta} \notag
     & \leq \frac{2R}{\pi}\left(1+\frac1{|\xi|}\right)\sqrt{2\beta}\notag\\
     & \leq \frac{2R}{\pi} \left(1+\frac1{\eta}\right)\sqrt{2\beta}.
\label{est|B|}
\end{align}

Now define
\[ \mu(dx) = \frac{g(x) 1_{[-R,R]}(x)\,\d x}{\int_{[-R,R]} g}\qquad
\nu(dx) = \frac{1_{[-R,R]}(x)\,\d x}{2R}.\]
By direct computation,
\begin{align*} H(\mu|\nu) &
= \int_{[-R,R]} g\log g + \log (2R) - \log \left(\int_{[-R,R]} g\right)\\
& \leq \int g|\log g| + \log(2R) - \log \left( 1-\frac1R^2\right),
\end{align*}
where we have used Chebyshev's inequality to get the bound on the last term:
$\int_{|x|>R} g \leq (1/R^2)\int g x^2\d x\leq 1/R^2$.
It is classical that $\int g|\log g|$ can be controlled by
$\int g\log g$ and $\int g x^2\d x=1$: indeed, if $\gamma$ stands
for the standard gaussian distribution, then
\[ \int_{g\leq 1} \Bigl(g \log \frac{g}{\gamma} - g + \gamma \Bigr)\geq 0,\]
so
\[ \int_{g\leq 1} g \log g \geq \int_{g\leq 1} g \log \gamma
+ \int_{g\leq 1} g - \int_{g\leq 1} \gamma;\]
replacing $\log\gamma$ by its explicit expression, we obtain
\begin{multline*} 
\int_{g\leq 1} g\log g\geq -\int_{g\leq 1} \frac{x^2}2 g(x)\,\d x
+ \left(1-\frac{\log(2\pi)}2\right) \int_{g\leq 1} g -
\int_{g\leq 1}\gamma \\
\geq -\left( 1 + \frac{\log(2\pi)}2\right).
\end{multline*}
The desired bound follows, since
\[ \int g|\log g| =\int g\log g - 2 \int_{g\leq 1} g\log g.\]
To summarize: there is an explicit bound
\begeq\label{HhR} H(\mu|\nu) \leq h(H,R).
\endeq

On the other hand, it follows from~\eqref{est|B|} that
\[ \nu[B]\leq \frac1\pi \left(1+\frac1\eta\right) \sqrt{2\beta}.\]
Combining this with~\eqref{HhR} and a general entropy inequality, we find
\[ \mu[B]\leq \frac{2 H(\mu|\nu)}{\log\left(1 + 
\frac{H(\mu|\nu)}{\nu[B]}\right)} \leq\frac{2h}
{\log \left(1+\frac{\pi h}{\left(1+\frac1\eta\right)\sqrt{2\beta}}\right)}.\]
So if $H$ and $\eta$ are given, there is a function $m(\beta)$,
going to~0 as $\beta\to 0$ and depending only on $H$ and $\eta$, such
that $\mu[B]\leq m(\beta)$.

The desired conclusion follows easily:
\begin{align*}
\int_{[-R,R]} g(x) \Bigl(1-\cos \bigl[2\pi(x+z)\xi]\Bigr)\,\d x &
\geq \beta \int_{[-R,R]} g(x)\,\d x \\
& \geq \beta \left(\int_{[-R,R]} g\right)\, \mu\bigl[\R\setminus [-R,R]\bigr]\\
& = \beta \left( 1 - \int_{|x|>R} g\right) \bigl( 1 - 
\mu[[-R,R]]\bigr)\\
& \geq \beta \left( 1 - \frac1{R^2}\right)
\left( 1 - \frac{2h}{\log\left(1+
\frac{\pi h}{\left(1+\frac1\eta\right)\sqrt{2\beta}}\right)}\right).
\end{align*}
This establishes (i) with
\[ \alpha:= \sup_{\beta,R}
\ \beta \left(1-\frac1{R^2}\right)
\left( 1 - \frac{2h}{\log\left(1+\frac{\pi h}{\left(1+\frac1{\eta}\right)
\sqrt{2\beta}}\right)}\right).\]

To get a lower bound on $\alpha$, one may choose for instance
\[ R=2,\quad
\beta = e^{-8h}\left(\frac{(\pi h)^2}{2(1+\eta^{-1})^2}\right),\]
then one finds
${\displaystyle  \alpha \geq \frac38 e^{-8h}\left(\frac{(\pi h)^2}{2(1+\eta^{-1})^2}
\right)}$.
\sm

Now let us prove $(ii)$. Assume for instance that $\xi>0$. By Taylor
formula,
\begin{align*}
e^{-2i\pi x\xi} & = 
1 - 2i\pi x \xi - 4 \pi^2 x^2 \int_0^\xi
(\xi-\zeta)\, e^{-2i\pi x\xi}\,\d \zeta\\
& = 1 - 2i\pi x\xi - 4\pi^2 x^2 \left(\frac{\xi^2}2\right)
+ 4 \pi^2 x^2 \int_0^\xi (\xi-\zeta) \bigl(1-e^{-2i\pi x\zeta}\bigr)\,\d\zeta.
\end{align*}
So for $|\xi|\leq\eta$, one has
\[ \Bigl|\hat{g}(\xi) - \bigl(1-2\pi^2\xi^2\bigr)\Bigr| \leq
\var \xi^2, \]
with
\begin{align*} \var & = \frac{4\pi^2}{\xi^2}
\left|\int_0^\xi (\xi-\zeta) \left( [1-e^{-2i\pi x\zeta}]x^2 g(x)\,\d x
\right)\,\d\zeta\right|\\
& \leq \frac{4\pi^2}{\xi^2}
\left(\int_0^\xi (\xi-\zeta)\,\d\zeta\right)\left(
\sup_{|\xi|\leq\eta} \int |1-e^{-2i\pi x\zeta}| x^2 g(x)\,\d x\right)\\
& = 4\pi^2 \sup_{|\zeta|\leq \eta} \int |\sin (\pi\zeta x)|\,x^2 g(x)\,\d x.
\end{align*}

For $|x|\leq 1/r$ and $|\zeta|\leq \eta$, one has
$|\sin(2\pi\zeta x)|\leq |2\pi\zeta x| \leq 2\pi\eta/r$;
on the other hand, for $|x|\geq 1/r$, we can use the trivial bound
$|\sin(\pi\zeta x)|\leq 1$. Therefore,
\begin{multline}
\int |\sin (\pi\zeta x)|\,x^2 g(x)\,\d x
\leq \frac{2\pi\eta}r + \int_{|x|\geq 1/r} x^2 g(x)\,\d x \\
\leq \frac{2\pi\eta}r + \int_{|x|\geq 1/r}x^2g(x)\,\d x
\leq \frac{2\pi\eta}r + \chi(r).
\end{multline}
In conclusion,
${\displaystyle \var \leq \inf_{r>0} \left[\frac{2\pi\eta}r + \chi(r)\right]}$,
and the right-hand side goes to~0 as $\eta\to 0$. This proves (ii).
\end{proof}

Now we can proceed with the main result of this Appendix.

\begin{Thm}[Local Central Limit Theorem]\label{LCLT}
Let $g$ be a probability distribution on $\R^k$, satisfying
\[ \int_{\R^k} g(x) x\,\d x =0, \qquad
\int_{\R^k} g(x) (x\otimes x)\,\d x = I_k,\qquad \int g\log g\d x \leq H,\]
where $I_k$ is the $k\times k$ identity matrix.
Let $\chi$ be such that $\chi(r)\to 0$ as $r\to 0$, and
${\displaystyle \int_{|x|\geq \frac1r} g(x)\,\d x \leq \chi(r)}$.
Let further $g_N(x)=\sqrt{N}^kg^{\ast N}(\sqrt{N}x)$, for any positive
integer $N$. Then:
\sm

(i) If $g\in L^p(\R^k)$, $1<p<\infty$, then $g_N$ is continuous
for $N\geq p'=p/(p-1)$; and for any $\delta>0$ there is
$\alpha=\alpha(\chi,H,\delta)>0$ and $\ov{\var}=
\ov{\var}(\chi,H,\delta)>0$ such that, given $k$, $\chi$ and $H$,
${\displaystyle \ov{\var}(\delta)\xrightarrow[\delta\to 0]{} 0}$
and, for $N\geq p'$,
\[ \sup_{x\in\R^k}
|g_N(x)-\gamma(x)| \leq \sqrt{N}
(1-\alpha)^{N-p'} \|g\|_{L^p}^{p'} \ + \ 
k\frac{e^{-2\pi^2 N\delta^2}}{\sqrt{N}\delta} \ + \ \ov{\var}(\delta),\]
where $\gamma$ stands for the standard Gaussian distribution.
In particular, $\sup |g_N-\gamma|$ goes to~0 as $N\to\infty$,
and there is an upper bound on the rate of convergence which only depends
on $k$, $p$, $\|g\|_{L^p}$, $\chi$ and $H$.
\sm

(ii) Given $\chi$ and $H$ there is a function $\lambda(N)$,
going to~0 as $N\to\infty$, such that if $g$ lies in $L^p$ for
some $p\in (1,\infty)$, then there is $N_0=N_0(\chi,H,N,p,\|g\|_{L^p})$
with
\[ N\geq N_0\Longrightarrow\quad
\sup_{x\in\R^k} |g_N-\gamma| \leq \lambda(N).\]
\end{Thm}

\begin{Rk} Part (ii) of this theorem is not used in this paper,
but it  is worth noticing in our ``entropic'' context:
It shows that there is a universal bound on the asymptotic rate of
convergence, depending only on energy and entropy estimates,
and independent of any $L^p$ bound. Still the $L^p$ bound provides
an estimate of the integer $N$ for which the estimate starts to be valid.
We do not know whether this information might be useful to obtain
appropriate versions of the Local Central Limit Theorem which do not
rely on $L^p$ estimates.
\end{Rk}

\begin{proof}[Proof of Theorem~\ref{LCLT}]
First, it follows from Young's convolution inequality that
$g^{\ast (N-1)}$ lies in $L^{p'}(\R^k)$; then its convolution product
with $g$ is continuous.

By properties of the Fourier transform,
\[ \hat{g_N}(\xi) = \hat{g} \left(\frac{\xi}{\sqrt{N}}\right)^N.\]
Without loss of generality, assume $p\leq 2$; then
$\hat{g}(\cdot/\sqrt{N})$ lies in $L^{p'}\cap L^\infty$ by the
Hausdorff-Young inequality, so $\hat{g_N}$ lies in $L^1$ for
$N\geq p'$. Then the Fourier inversion formula applies:
\[ g_N(x) = \int_{\R^k} \hat{g_N}(\xi) e^{2i\pi x\cdot\xi}\,\d\xi.\]
In particular, for any $x\in\R$,
\[ |g_N(x)-\gamma(x)| = \left|
\int_{\R^k} \bigl(\hat{g_N}(\xi)-\hat{\gamma}(\xi)\bigr) 
e^{2i\pi x\cdot\xi}\,\d \xi \right| \leq \int |\hat{g_N}-\hat{\gamma}|.\]
We separate between low and high frequencies, according to some
threshold $\delta \sqrt{N}$, to choose later:
\begeq\label{boundgNg}
\sup_{x\in\R^k} |g_N(x)-\gamma(x)|
\leq \int_{|\xi|>\delta\sqrt{N}}|\hat{g_N}| \ 
+ \ \int_{|\xi|>\delta\sqrt{N}} |\hat{\gamma}| \
+ \ \int_{|\xi|\leq \delta\sqrt{N}} |\hat{g_N}-\hat{\gamma}|.
\endeq

To estimate the first term in the right-hand side of~\eqref{boundgNg},
we use Proposition~\ref{propFourier}~(i):
there is $\alpha=\alpha(\delta,\chi,H)>0$ such that
$|\xi|\geq \eta\Longrightarrow |g(\xi)|\leq 1-\alpha$; so
\[\int_{|\xi|>\delta\sqrt{N}} |\hat{g_N}| = 
\int_{|\xi|>\delta\sqrt{N}} \left| \hat{g}\left(\frac{\xi}{\sqrt{N}}\right)
\right|^N\,\d \xi = \sqrt{N}^k \int_{|\xi|>\delta} |\hat{g}|^N\,\d\xi\]
\[ \leq \sqrt{N}^k (1-\alpha)^{N-p'} \int |\hat{g}|^{p'}.\]
Combining this with the Hausdorff-Young inequality, we find
\begeq\label{HF1} \int_{|\xi|>\delta\sqrt{N}} |\hat{g_N}|
\leq \sqrt{N}^k\, (1-\alpha)^{N-p'} \|g\|_{L^p}^{p'}.
\endeq

The second term in the right-hand side of~\eqref{boundgNg} can be bounded
by an explicit estimate:
\[ \int_{|\xi|>\delta\sqrt{N}} |\hat{\gamma}| \leq
k\frac{e^{-2\pi^2 N\delta^2}}{\delta\sqrt{N}}.\]

The third term in the right-hand side of~\eqref{boundgNg} is a bit more
subtle. On one hand, by a telescopic sum argument, since
$|\hat{g}|\leq 1$ and $|\hat{\gamma}|\leq 1$, we have
\[ | \hat{g_N}(\xi) - \hat{\gamma}(\xi)|
= \left| \hat{g}\left(\frac{\xi}{\sqrt{N}}\right)^N  - 
\hat{\gamma}\left(\frac{\xi}{\sqrt{N}}\right)^N \right|\leq
N \left| \hat{g}\left(\frac{\xi}{\sqrt{N}}\right)
- \hat{\gamma} \left(\frac{\xi}{\sqrt{N}}\right)\right|.\]
Now we can apply Proposition~\ref{propFourier} (ii), with $\xi$ replaced
by $\xi/\sqrt{N}$ and $\var=\max(\var_g,\var_\gamma)$ where
$\var_g$ is the function appearing in the Proposition:
\begin{multline} \label{LF1}
 \left| \hat{g}\left(\frac{\xi}{\sqrt{N}}\right) - \hat{\gamma}(\xi)\right|
\leq N \left| \hat{g}\left(\frac{\xi}{\sqrt{N}}\right)
- \left( 1-\frac{2\pi^2|\xi|^2}{N}\right)\right|+ \\
N \left| \hat{\gamma}\left(\frac{\xi}{\sqrt{N}}\right)
- \left( 1-\frac{2\pi^2|\xi|^2}{N}\right)\right| \leq 
2N \var(\delta) \left(\frac{\xi}{\sqrt{N}}\right)^2 = 2\,\var(\delta)\, \xi^2.
\end{multline}
Here $\var$ is a function depending only on $H$ and $\chi$.

On the other hand,
by Proposition~\ref{propFourier} (iii),
\[ |\hat{g_N}(\xi)-\hat{\gamma}(\xi)|
\leq \left|\hat{g}\left(\frac{\xi}{\sqrt{N}}\right)\right|^N
+ \hat{\gamma}(\xi) \leq
\max \left( 1-\frac{\pi^2|\xi|^2}{N},\ 1-\alpha_0\right)^N +
\hat{\gamma}(\xi).\]
Thanks to the inequality $(1-u/N)^N\leq e^{-u}$, we conclude that
\begin{multline}\label{LF2}
|\hat{g_N}(\xi)-\hat{\gamma}(\xi)|\leq
\max\Bigl(e^{-\pi^2|\xi|^2},\ (1-\alpha_0)^N\Bigr) +\hat{\gamma}(\xi) \\
\leq 2\max\Bigl(e^{-\pi^2|\xi|^2},\ (1-\alpha_0)^N\Bigr).
\end{multline}

By taking the geometric mean of~\eqref{LF1} and~\eqref{LF2}, we obtain
\[ \left| \hat{g}(\xi)-\hat{\gamma}(\xi)\right|
\leq \sqrt{2\var}\, |\xi| \max \Bigl( e^{-\frac{\pi^2|\xi|^2}2},\ 
(1-\alpha_0)^\frac{N}2\Bigr).\]
Then
\begin{align*} \int_{|\xi|\leq \delta\sqrt{N}}
\left| \hat{g}(\xi)-\hat{\gamma}(\xi)\right|
& \leq \sqrt{2\var} 
\left( \int_{|\xi|\leq\delta\sqrt{N}} |\xi|\,
e^{-\frac{\pi^2|\xi|^2}2}\,\d\xi \ + \ 
(1-\alpha_0)^{\frac{N}2}\int_{|\xi|\leq \delta\sqrt{N}}
|\xi|\,\d\xi\right) \\
& \leq \sqrt{2\var} \left( \int |\xi|\,
e^{-\frac{\pi^2|\xi|^2}2}\,\d\xi\ + \
(1-\alpha_0)^{\frac{N}2} |S^{k-1}|\frac{(\delta\sqrt{N})^{k+1}}{k+1} \right)\\
& \leq \sqrt{2\var} \left( \int |\xi|\,
e^{-\frac{\pi^2|\xi|^2}2}\,\d\xi\ + C(k,\alpha_0)\right)\
=:\ov{\var}(\delta,\chi,H)
\end{align*}
(where $C(k,\alpha_0)$ is a constant which does not depend on $N$).
This concludes the proof of (i).

To prove (ii), we let $p$ vary with $N$ in such a way that
$p$ remains of the order of $N$; for instance, $p'=N/2$
(for $N$ large enough). Then, as $N$ goes to infinity,
\[ \log \|g\|_{L^p}^{p'} =
\frac{1}{p-1} \log \int g^p\ \xrightarrow[p\to 1]{}\ \int g\log g.\]
So, when $N$ is large enough, the first term in the right-hand side
of (i) can be bounded by $(1-\alpha)^{N/2}e^H$, which does not
depend on the $L^p$ norm of $g$. (But ``large enough'' here may
depend on this norm!!)
\end{proof}

\Ack

\signec
\signmcc
\signjl
\signml
\signcv

\end{document}